\documentclass[12pt,a4paper,oneside,draft]{amsart}

\usepackage{amsmath}
\usepackage{amscd}
\usepackage{amssymb}

\long\def\symbolfootnote[#1]#2{\begingroup \def\thefootnote{\fnsymbol{footnote}}\footnote[#1]{#2}\endgroup}

\swapnumbers
\newcommand\mcnt{subsection}
\newtheorem{theorem}[\mcnt]{Theorem}

\newtheorem{lemma}[\mcnt]{Lemma}
\newtheorem{corollary}[\mcnt]{Corollary}

\newtheorem{proposition}[\mcnt]{Proposition}

\newtheorem{remark}[\mcnt]{Remark}
\newtheorem{example}[\mcnt]{Example}

\newtheorem{definition}[\mcnt]{Definition}

\makeatletter
\@addtoreset{subsubsection}{section}
\@addtoreset{subsubsection}{subsection}
\@addtoreset{subsection}{section}
\@addtoreset{equation}{section}
\makeatother

\providecommand\eqref[1]{(\ref{#1})}

\newcommand\NNN{{\mathbb N}}
\newcommand\RRR{{\mathbb R}}

\newcommand\ZZZ{{\mathbb Z}}

\newcommand\BB{{\mathcal B}}

\newcommand\NN{{\mathcal N}}

\newcommand\XX{{\mathcal X}}
\newcommand\YY{{\mathcal Y}}

\newcommand\id{\mathrm{id}}

\newcommand\grad{\triangledown} 

\newcommand\Stab{\mathcal{S}}
\newcommand\Diff{\mathcal{D}}

\newcommand\End{{\mathcal{E}}}

\newcommand\eps{\varepsilon}

\newcommand\AFlow{\mathbf{G}}
\newcommand\AFld{G}
\newcommand\BFlow{\mathbf{F}}
\newcommand\BFld{F}
\newcommand\HFld{H}

\newcommand\Dom[1]{\mathsf{dom}(#1)}

\newcommand\BDom{\Dom{\BFlow}}

\newcommand\Shift{\varphi}
\newcommand\hShift{\hat\Shift}

\newcommand\ShBV{\Shift_{\Vman}}

\newcommand\BFunc{f}
\newcommand\GFunc{\Gamma}
\newcommand\PFunc{P}

\newcommand\Pmap{Q}
\newcommand\Qmap{P}

\newcommand\Hmap{H}

\newcommand\difM{h}

\newcommand\gdifM{\tau}

\newcommand\afunc{\alpha}
\newcommand\bfunc{\beta}
\newcommand\cfunc{\gamma}

\newcommand\vfunc{v}

\newcommand\dg{p}

\newcommand\Vman{V}

\newcommand\Nbh{\mathcal{U}}

\newcommand\EndFlowNbh[2]{\End(#1,#2)}

\newcommand\EndBV{\EndFlowNbh{\BFlow}{\Vman}}

\newcommand\EndBVi[1]{\End_{#1}(\BFlow,\Vman,\orig)}

\newcommand\gEnd{\hat{\Diff}}
\newcommand\gEndId{\gEnd_{\id}}

\newcommand\gEndBz{\gEnd(\BFlow)}

\newcommand\gEndIdBzr[1]{\gEndId(\BFlow)^{#1}}

\newcommand\Poly[1]{\mathcal{P}_{#1}}

\newcommand\singA{\Sigma_{\AFld}}
\newcommand\singB{\Sigma_{\BFld}}

\newcommand\ord{\mathrm{ord}}

\newcommand\AST{{\rm($\ast$)}}

\newcommand\Nman{N}

\newcommand\EndBViz[1]{X_{\Qmap}} 

\newcommand\orig{O}

\newcommand\GL{\mathrm{GL}}
\newcommand\Lmatr{L}
\newcommand\Umatr{U}

\newcommand\Asubm{V}

\newcommand\rsm{r}
\newcommand\contWW[2]{C^{#1,#2}_{W,W}}

\newcommand\Bsh[1]{\BFlow_{#1}}
\newcommand\hBsh[1]{\BFlow_{#1}}
\newcommand\qBsh[2]{\BFlow_{#1,#2}}
\newcommand\hBshh[1]{\qBsh{\difM}{#1}}

\newcommand\gimShB{\hat{Sh}(\BFlow)}
\newcommand\jo[1]{j^{#1}}
\newcommand\jequ[3]{\jo{#3}(#1) =\jo{#3}(#2)}

\newcommand\orb{\gamma}

\newcommand\Mman{M}
\newcommand\funcBV{\mathsf{func}(\BFlow,\Vman)}
\newcommand\funcB{\mathsf{func}(\BFlow)}
\newcommand\imShBV{Sh(\BFlow,\Vman)}
\newcommand\imShB{Sh(\BFlow)}
\newcommand\EidBVr[1]{\mathcal{E}_{\id}(\BFlow,\Vman)^{#1}}
\newcommand\EidBr[1]{\mathcal{E}_{\id}(\BFlow)^{#1}}

\newcommand\EBV{\mathcal{E}(\BFlow,\Vman)}

\newcommand\ShB{\Shift}

\newcommand\EBVi{\End_{\infty}(\BFlow,\Vman)}

\newcommand\DiffRno{\hat{\mathcal{D}}(\RRR^n)}
\newcommand\EndRno{\hat{\mathcal{E}}(\RRR^n)}
\newcommand\FuncRno{\hat{\mathcal{F}}(\RRR^n)}

\newcommand\EndRnoR{\hat{\mathcal{E}}(\RRR^n,\orig;\RRR^1)}
\newcommand\EndRnoRno{\hat{\mathcal{E}}(\RRR^n,\orig;\RRR^n,\orig)}
\newcommand\EndRnoRm{\hat{\mathcal{E}}(\RRR^n,\orig;\RRR^m)}
\newcommand\EndRnoRmo{\hat{\mathcal{E}}(\RRR^n,\orig;\RRR^m,\orig)}

\newcommand\JSh{\mathcal{J}}
\newcommand\hJSh{\hat{\JSh}}
\newcommand\grp{\mathcal{G}}

\newcommand\gJShBO{\hJSh(\BFlow)}
\newcommand\gJShBOi[1]{\hJSh_{#1}(\BFlow)}

\newcommand\JShBV{\JSh(\BFlow,\Vman)}
\newcommand\JShBVi[1]{\JSh_{#1}(\BFlow,\Vman)}

\newcommand\DiffMz{\hat{\mathcal{D}}(\Mman,z)}
\newcommand\EndMz{\hat{\mathcal{E}}(\Mman,z)}
\newcommand\EndMzM{\hat{\mathcal{E}}(\Mman,z;\Mman)}
\newcommand\FuncMz{\hat{\mathcal{F}}_{z}(\Mman)}

\newcommand\gEBz{\gEnd(\BFlow,z)}
\newcommand\gEidBzr[1]{\gEndId(\BFlow,z)^{#1}}
\newcommand\gimShBz{\hat{Sh}(\BFlow,z)}

\newcommand\gEBO{\gEnd(\BFlow)}
\newcommand\gEidBOr[1]{\gEndId(\BFlow)^{#1}}
\newcommand\gimShBO{\hat{Sh}(\BFlow)}

\newcommand\funcBM{\mathsf{func}(\BFlow,\Mman)}
\newcommand\imShBM{Sh(\BFlow,\Mman)}
\newcommand\EidBMr[1]{\mathcal{E}_{\id}(\BFlow,\Mman)^{#1}}

\newcommand\coA{{\rm(A1)}}
\newcommand\coB{{\rm(A2)}}
\newcommand\coC{{\rm(A3)}}

\newcommand\JRnok{\hat{J}^{k}(\RRR^n)}
\newcommand\JRnoi{\hat{J}^{\infty}(\RRR^n)}

\newcommand\coiA{{\rm(G1)}}
\newcommand\coiB{{\rm(G2)}}
\newcommand\coiC{{\rm(G3)}}

\newcommand\SCT{\Lambda}
\newcommand\EJ{\mathcal{H}}
\newcommand\SI{\Psi}
\newcommand\RS{\sigma}
\newcommand\JS{\SCT}

\begin{document}
\title[$\infty$-jets of orbit preserving diffeomorphisms]
{$\infty$-jets of diffeomorphisms preserving orbits of vector fields}
\author{Sergiy Maksymenko}

\begin{abstract}
Let $F$ be a $C^{\infty}$ vector field defined near the origin $O\in\mathbb{R}^{n}$, $F(O)=0$, and $(\mathbf{F}_{t})$ be its local flow.
Denote by $\hat{\mathcal E}(F)$ the set of germs of orbit preserving diffeomorphisms $h:\mathbb{R}^{n}\to\mathbb{R}^{n}$ at $\orig$, and let $\hat{\mathcal E}_{\mathrm{id}}(F)^{r}$, $(r\geq0)$, be the identity component of $\hat{\mathcal E}(F)$ with respect to $C^{r}$ topology.
Then $\hat{\mathcal E}_{\mathrm{id}}(F)^{\infty}$ contains a subset $\hat{Sh}(F)$ consisting of maps of the form $\mathbf{F}_{\alpha(x)}(x)$, where $\alpha:\mathbb{R}^n\to\mathbb{R}$ runs over the space of all smooth germs at $\orig$.
It was proved earlier by the author that if $F$ is a linear vector field, then 
$\hat{Sh}(F) = \hat{\mathcal E}_{\mathrm{id}}(F)^{0}.$

In this paper we present a class of examples of vector fields with degenerate singularities at $O$ for which $\hat{Sh}(F)$ \emph{formally} coincides with $\hat{\mathcal E}_{\mathrm{id}}(F)^{1}$, i.e. \emph{on the level of $\infty$-jets at $\orig$}.

We also establish parameter rigidity of linear vector fields and ``reduced'' Hamiltonian vector fields of real homogeneous polynomials in two variables.
\end{abstract}
\maketitle


{\em AMSClass:} 37C10 

\protect\vskip-3cm
\section{Introduction}\label{sect:intro}
Let $\BFld$ be a smooth ($C^{\infty}$) vector field on a smooth manifold $\Mman$, $(\BFlow_t)$ be the local flow generated by $\BFld$, and $\singB$ be the set of singular points of $\BFld$.
In this paper we consider smooth maps $\difM:\Mman\to\Mman$ preserving the (singular) foliation on $\Mman$ by orbits of $\BFld$, i.e.\! $\difM(\Mman\cap \orb)\subset\orb$ for every orbit $\orb$ of $\BFld$.

The groups of leaf preserving diffeomorphisms and homeomorphisms of foliations are intensively studied.
Most of the results concern with regular foliations, see e.g.~\cite{Banyaga:T:1977,Rybicki:SJM:1996,Rybicki:DM:1996,AbeFukui:CEJM:2003} and references in these papers.
For singular foliations the situation is much more difficult.
Therefore usually foliations by orbits of actions of finite-dimensional Lie groups are considered, e.g.~\cite{Schwarz:T:1975,Mather:T:1977, AbeFukui:T:2001, Rybicki:TA:2007}.
Homeomorphisms preserving foliations of vector fields are studied e.g. in~\cite{CamachoNeto:LNM:1977, GutierrezMelo:1977}.

The approach used in this paper is specific for the case of flows.
By definition for every $x\in\Mman$ its image $\difM(x)$ belongs to the orbit $\orb_x$ of $x$.
Therefore we want \emph{to associate to $x$ the time $\afunc_{\difM}(x)$ between $x$ and $\difM(x)$ along $\orb_x$}, so that 
\begin{equation}\label{equ:h_is_a_shift}
  \difM(x) = \BFlow_{\afunc_{\difM}(x)}(x).
\end{equation}
We will call $\afunc_{\difM}$ a \emph{shift function} for $\difM$, which in turn will be called the \emph{shift} along orbits of $\BFld$ via $\afunc_{\difM}$.

Such ideas were used e.g.\! in~\cite{EHoph:1937,Chacon:JMM:1966,Totoki:MFCKUS:1966, Kowada:JMSJ:1972,Parry:JLMS:1972,Kochergin:IANSSSR:1973} and others for reparametrizations of measure preserving flows and study their mixing properties.
In these papers $\afunc_{\difM}$ is required to be measurable, so it can even be discontinuous and its values on subsets of measure $0$ can be ignored.
In~\cite{OrnsteinSmorodsky:IJM:1978} continuity of shift functions was investigated.
In contrast, we work in $C^{\infty}$ category and require that \emph{$\afunc_{\difM}$ is $C^{\infty}$ whenever so is $\difM$}.
One of the main problems here is to define $\afunc_{\difM}$ near a singular point of a vector field, see~\ref{subsect:sigma_def_pt}-\ref{subsect:h_dif_at_fix}.

Smooth shift functions were used in authors papers~\cite{Maks:TA:2003,Maks:AGAG:2006,Maks:BSM:2006} for calculations of homotopy types of certain infinite-dimensional spaces.
The present paper brings another application of shift functions to smooth reparametrization of flows and in particular to parameter rigidity.

In~\cite{Maks:TA:2003} the problem of finding representation~\eqref{equ:h_is_a_shift} was solved for linear flows.
It was shown that if $\BFld$ is a linear vector field on $\RRR^n$, then for every diffeomorphism $\difM:\RRR^n\to\RRR^n$ preserving orbits of $\BFld$ and being isotopic to the identity map $\id_{\RRR^n}$ via an orbit preserving isotopy there exists a $C^{\infty}$ shift function $\afunc_{\difM}$.
Moreover, if a family $\difM_{s}$ of orbit preserving diffeomorphisms smoothly depends on some $k$-dimensional parameter $s$, then so does the family $\afunc_{\difM_s}$ of their shift functions%
\protect\symbolfootnote[2]{
I must warn the reader that my paper~\cite{Maks:TA:2003} contains mistakes in the estimations of continuity of the correspondence $\difM\mapsto\afunc_{\difM}$ regarded as a map between certain functional spaces.
In particular in~\cite[Defn.~15]{Maks:TA:2003} it should be additionally required that the image $\varphi_{V}(\mathcal{M})$ is at least $C^{\infty}_{W}$-open in the image of the map $\varphi_{V}$.
Without this assumption~\cite[Th.~17]{Maks:TA:2003} is not true.
Moreover in \cite[Lm.~31]{Maks:TA:2003} the mapping $Z^{-1}$ is $C^{r+1,r}_{W,W}$-continuous in the real case and only $C^{\infty,\infty}_{W,W}$-continuous in the complex case.
As a result the formulations of~\cite[Th~1, Th.27 (part concerning (S)-points) \& Lm.~28]{Maks:TA:2003} should be changed.

Unfortunately  \cite[Th.~1]{Maks:TA:2003} was essentially used in~\cite{Maks:AGAG:2006} for the calculations of the homotopy types of stabilizers and orbits of Morse functions on surfaces.
We will repair the mistakes of~\cite{Maks:TA:2003} in another paper and show that the part of results~\cite{Maks:TA:2003} used in~\cite{Maks:AGAG:2006} remains true.

Also notice that the formula~\cite[Eq.~(10)]{Maks:TA:2003} for the shift functions at regular points is misprinted.
It must be read as follows:
$\alpha(x) = p_1\circ f(x) - p_1\circ \Phi(x,a)+a.$
}\label{foot:1}.

Our first result claims that the last two properties easily imply parameter rigidity of a vector field, see Theorem~\ref{th:param-rigid}.
In particular, as a consequence of~\cite{Maks:TA:2003}, we obtain parameter rigidity of linear vector fields and their regular extensions.
Notice that this statement together with the result of S.\;Sternberg~\cite{Sternberg:AJM:1957} implies parameter rigidity of a large class of ``hyperbolic'' flows, which agrees with discovered about thirty years ago rigidity of locally free hyperbolic actions of certain Lie groups~\cite{KatokSpatzier:PMIHES:1994}.
Though we consider actions of the one-dimensional group $\RRR$ only, Theorem~\ref{th:param-rigid} is nevertheless stronger in the part that we admit fixed points, i.e.\! non-locally free actions.

Further we deal with the situation when $\BFld$ is a vector field defined on some neighbourhood $\Vman$ of the origin $\orig\in\RRR^n$ being its singular point.
Denote by $\gimShBO$ the group of germs at $\orig$ of smooth shifts, i.e. maps of the form~\eqref{equ:h_is_a_shift}.
In \S\ref{sect:charact_inf_jets} we describe the structure of $\infty$-jets of elements of $\gimShBO$ and establish a necessary and sufficient condition for a certain group $\grp$ of germs of diffeomorphisms of $(\RRR^n,\orig)$ to coincide with $\gimShBO$ on the level of $\infty$-jets at $\orig$, see Theorem~\ref{th:grp_subset_JSh}.

Let $\JShBV$ be the space of all smooth maps $\difM:\Vman\to\RRR^{n}$ whose $\infty$-jet $\jo{\infty}(\difM)$ at $\orig$ coincides with the $\infty$-jet of some shift $\BFlow_{\afunc_{\difM}}$, where $\afunc_{\difM}\in C^{\infty}(\Vman,\RRR)$.
Though in general such a function $\afunc_{\difM}$ is not unique, we show in \S\S\ref{sect:Borel_th},\ref{sect:cont-corresp} that it can be chosen so that the correspondence $\difM\mapsto\afunc_{\difM}$ becomes a continuous (and in a certain sense smooth) map $\JS:\JShBV \supset \XX\to C^{\infty}(\Vman,\RRR)$ defined on some subset $\XX$ of $\JShBV$ (Theorem~\ref{th:cont-estimations}).
Actually \S\ref{sect:Borel_th} contains a variant of a well-known theorem of E.~Borel about smooth functions with given Taylor series (Theorem~\ref{th:Borel}).
This theorem is then used in \S\ref{sect:cont-corresp} for the construction of $\JS$.
In \S\ref{sect:appl_jinf_sect} we also present application of $\JS$ to the problem of resolving~\eqref{equ:h_is_a_shift}.

Denote by $\gEBO$ the group of germs of orbit preserving diffeomorphisms for $\BFld$, and let $\gEidBOr{1}$ be its path component with respect to weak $C^{1}_{W}$ topology.
In \S\ref{sect:prop-AST} we introduce a certain condition \AST\ on $\BFld$ guaranteeing that $\gimShBO$ coincides with $\gEidBOr{1}$ on the (formal) level of $\infty$-jets, see Theorem~\ref{th:AST_implies_infjets}.
The proof of this theorem is given in \S\S\ref{sect:prelim:th:AST_implies_infjets},\ref{sect:th:AST_implies_infjets}.

Finally, in \S\ref{sect:vf_on_r2} we present a class of vector fields on $\RRR^2$ satisfying condition \AST\ and explain that for these vector fields $\gimShBO=\gEidBOr{1}$.
This improves results of~\cite{Maks:Hamvf:2006}, which were based on a previous (unpublished) version of this paper (Theorem~\ref{th:HamVF}).

In another paper the last theorem will be used to extend calculations of~\cite{Maks:AGAG:2006} to a large class of functions with degenerate singularities on surfaces.

\subsection{Preliminaries}
Let $A$ and $B$ be smooth manifolds.
Then for every $r=0,1,\ldots,\infty$ we can define the weak $C^{r}_{W}$ topology on $C^{\infty}(A,B)$, see e.g.~\cite{GolubitskyGuillemin,Hirsch:DiffTop}.
We will assume that the reader is familiar with these topologies.
It easily follows from definition that topology $C^{0}_{W}$ coincides with the \emph{compact open} one.
Moreover, let $J^{r}(A,B)$ be the manifold of $r$-jets of $C^{r}$ maps $A\to B$.
Associating to every $\difM\in C^{\infty}(A,B)$ its $r$-jet extension being an element of $J^{r}(A,B)$, we obtain a natural inclusion $C^{\infty}(A,B) \subset C^{r}(A,J^{r}(A,B))$.
Then $C^{r}_{W}$ topology on $C^{\infty}(A,B)$ can be defined as the topology induced by $C^{0}_{W}$ topology of $C^{r}(A,J^{r}(A,B))$.

We say that a subset $\XX\subset C^{\infty}(A,B)$ is \emph{$C^{k}_{W}$-open\/} if it is open with respect to the induced $C^{k}_{W}$-topology of $C^{\infty}(A,B)$.
\begin{definition}\label{defn:r-homotopies}
Let $\Hmap:A\times I \to B$ be a homotopy such that for every $t\in I$ the mapping $\Hmap_t:A\to B$ is $C^{r}$.
We will call $\Hmap$ an {\bfseries $r$-homotopy} if the following map
$$j^{r}\Hmap: A \times I \to J^{r}(A, B), 
\qquad (a,t) \mapsto j^{r}(\Hmap_t)(a),$$ 
associating to every $(a,t)\in A \times I$ the $r$-th jet of $\Hmap_t$ at $a$, is continuous.
In local coordinates this means that $\Hmap_t$ and all its partial derivatives ``along $A$'' are continuous in $(a,t)$.
In particular, every $C^{r}$-homotopy is an $r$-homotopy as well.

Equivalently, regarding a homotopy $\Hmap$ as a path $\hat\Hmap:I\to C^{\infty}(A,B)$ defined by $\hat\Hmap(t)(a)=\Hmap(a,t)$, we see that $\Hmap$ is an {\bfseries $r$-homotopy} if and only if $\hat\Hmap$ is a continuous path into $C^{r}_{W}$-topology of $C^{\infty}(A,B)$.

If $\Hmap$ is an $r$-homotopy consisting of embeddings, it will be called an {\bfseries $r$-isotopy}.
\end{definition}

Let $C$ and $D$ be some other smooth manifolds and $\YY\subset C^{\infty}(C,D)$ be a subset.
A map $u:\XX\to\YY$ will be called  \emph{$C^{s,r}_{W,W}$-continuous} if it is continuous from $C^{s}_{W}$-topology of $\XX$ to $C^{r}_{W}$-topology of $\YY$, $(r,s=0,1,\ldots,\infty)$.

\begin{definition}\label{defn:pres_smoothness}
We will say that $u:\XX\to\YY$ {\bfseries preserves smoothness} if for any $C^{\infty}$ map $H:A\times \RRR^n\to B$ such that $H_t=H(\cdot,t)\in \XX$ for all $t\in\RRR^n$ the following mapping
$$
u(H): C\times \RRR^n\to D, 
\qquad 
u(H)(c,t) = u(H_t)(c)
$$
is $C^{\infty}$ as well.
\end{definition}

\section{Obstructions for shift functions}\label{sect:shift-map}
In this section we briefly discuss obstructions for smooth resolvability of~\eqref{equ:h_is_a_shift}.
\subsection{}\label{subsect:sigma_def_pt}
\emph{Evidently, the value $\afunc_{\difM}(z)$ is uniquely defined only if $z$ is regular and non-periodic for $\BFld$.
If $z$ is a periodic point of period $\theta$, then $\afunc_{\difM}(z)$ is defined only up to a constant summand $n\theta$, $(n\in\ZZZ)$, while if $z$ is fixed, we can set $\afunc_{\difM}(z)$ to arbitrary number.}

In applications to ergodic flows this problem usually does not appear: the union of periodic and fixed points is an invariant subset, therefore it can be assumed to have measure $0$.
Hence the values of $\afunc_{\difM}$ on this set may be ignored.
Sometimes it is also assumed that the set of periodic points is empty, e.g.~\cite[p.357]{Kowada:JMSJ:1972}.

\subsection{}\label{subsect:sigma_smooth_at_reg}
\emph{If $z$ is a regular (even periodic) point of $\BFld$, then $\afunc_{\difM}$ can be smoothly defined on some neighbourhood of \begin{large}                                                                                                                                   \end{large}$z$}, see~\cite[\S3.1]{Maks:TA:2003} and footnote on page~\pageref{foot:1} for correct reading of ~\cite[Eq.~(10)]{Maks:TA:2003}.

\subsection{}\label{subsect:h_isot_id}
\emph{Representation~\eqref{equ:h_is_a_shift} with smooth $\afunc_{\difM}$ implies that $\difM$ is homotopic to the identity $\id$ via a smooth orbit preserving homotopy}.
For instance we can take the following one: $\difM_t(x)=\BFlow(x,t\,\afunc_{\difM}(x))$.

\subsection{}\label{subsect:sigma_def_reg_h_isot}
Conversely, \emph{if $\difM$ is homotopic to $\id$ via some smooth orbit preserving homotopy, then (using this homotopy and~\ref{subsect:sigma_smooth_at_reg}) we can smoothly define $\afunc_{\difM}$ on the set of regular points of $\BFld$, see~\cite[Th.~25]{Maks:TA:2003}, though $\afunc_{\difM}$ can even be discontinuous at singular points of $\BFld$}, 
In general, $\afunc_{\difM}$ depends on a particular homotopy.

\subsection{}\label{subsect:h_dif_at_fix}
\emph{If $\afunc_{\difM}$ can be defined at some singular point $z$ of $\BFld$ so that it becomes smooth near $z$, then $\difM$ must be an embedding at $z$, \cite[Cor.~21]{Maks:TA:2003}}.

\section{Shift map}\label{sect:shift_map}
Observations of the previous section lead to the following construction of shift map used in~\cite{Maks:TA:2003}.

Let $\Mman$ be a smooth manifold and $\BFld$ be a vector field on $\Mman$ tangent to $\partial\Mman$.
Then for every $x\in\Mman$ its orbit with respect to $\BFld$ is a unique mapping
$\orb_x: \RRR\supset(a_x,b_x) \to \Mman$ such that $\orb_x(0)=x$ and $\dot{\orb}_x=\BFld(\orb_x)$, where $(a_x,b_x) \subset\RRR$ is the maximal interval on which a map with the previous two properties can be defined.
Then
$$
\BDom = \mathop\cup\limits_{x\in\Mman} x \times (a_x, b_x),
$$
is an open neighbourhood of $\Mman\times0$ in $\Mman\times\RRR$, and by definition the \emph{local flow\/} of $\BFld$ is the following map
$$\BFlow: \Mman\times\RRR \; \supset \;\BDom\longrightarrow\Mman,
\qquad
\BFlow(x,t) = \orb_x(t).
$$
If $\Mman$ is compact, or more generally if $\BFld$ has compact support, then $\BDom=\Mman\times\RRR$ and thus $\BFlow$ is \emph{global}, i.e. is defined on all of $\Mman\times\RRR$, see e.g.~\cite{PalisdeMelo}.

For every open $\Vman \subset \Mman$ denote by $\funcBV$ the subset of $C^{\infty}(\Vman,\RRR)$ consisting of functions $\afunc$ whose graph $\Gamma_{\afunc}=\{(x,\afunc(x)) \ : \ x\in\Vman\}$ is contained in $\BDom$.
Then we can define the following map
\begin{equation}\label{equ:glob_shift_map}
\begin{array}{c}
\ShBV: C^{\infty}(\Vman,\RRR) \;\supset\;  \funcBV\; \longrightarrow \;C^{\infty}(\Vman,\Mman), \\[2mm]
\ShBV(\afunc)(x) = \BFlow(x,\afunc(x)),
\end{array}
\end{equation}
which will be called the \emph{shift map} of $\BFld$ on $\Vman$.
Its image in $C^{\infty}(\Vman,\Mman)$ will be denoted by $\imShBV$.
If $\BFlow$ is global, then $\funcBV=C^{\infty}(\Vman,\RRR)$.

It is easy to see that $\ShBV$ is $C^{r,r}_{W,W}$-continuous for all $r=0,1,\ldots,\infty$, \cite[Lemma~2]{Maks:TA:2003}.
Moreover, if the set $\singB$ of singular points of $\BFld$ is nowhere dense in $\Vman$, then $\ShBV$ is locally injective with respect to any $C^{r}_{W}$ topology of $\funcBV$, \cite[Prop.~14]{Maks:TA:2003}.

Denote by $\EBV \subset C^{\infty}(\Vman,\Mman)$ the subset consisting of all smooth maps $\difM:\Vman\to\Mman$ such that
\begin{itemize}
\item
$\difM(\omega\cap\Vman)\subset\omega$ for every orbit $\omega$ of $\BFld$, in particular $\difM$ is fixed on $\singB\cap\Vman$, and
\item
$\difM$ is a local diffeomorphism at every $z\in\singB\cap\Vman$.
\end{itemize}
Let $\EidBVr{r}$, $(0\leq r\leq \infty)$, be the path component of the identity embedding $i_{\Vman}:\Vman\subset\Mman$ in $\EndBV$ with respect to $C^{r}_{W}$-topology, i.e. $\EidBVr{r}$ consists of all smooth maps $\difM:\Vman\to\Mman$ which are $r$-homotopic to $i_{\Vman}$ in $\EBV$.

Then we have the following inclusions:
\begin{equation}\label{equ:inclusions_sh_concmp}
\imShBV \subset \EidBVr{\infty} \subset\cdots\subset \EidBVr{1} \subset \EidBVr{0}.
\end{equation}
The first one follows from~\ref{subsect:h_isot_id} and~\ref{subsect:h_dif_at_fix}, and all others are evident.

For $\Vman=\Mman$ we denote
\ $\Shift=\Shift_{\Mman}$, 
\ $\funcB=\funcBM$, 
\ $\imShB=\imShBM$, and
\ $\EidBr{r}=\EidBMr{r}$.

\subsection{Local shift map.}
For $z\in\Mman$ let $\FuncMz$ be the algebra of germs at $z$ of smooth functions $\Mman\to\RRR$, $\EndMzM$ be the space of germs at $z$ of all smooth maps $f:\Mman\to\Mman$, $\EndMz \subset \EndMzM$ be the subset consisting of germs $f$ such that $f(z)=z$, and $\DiffMz\subset\EndMz$ be the subset consisting of germs of all diffeomorphisms.

We want to define the following \emph{local shift map} analogous to~\eqref{equ:glob_shift_map}:
$$
\hShift:\FuncMz\to \EndMzM,
\qquad
\hShift(\afunc)(x) = \BFlow(x,\afunc(x)).
$$
If $\BFlow$ is not global, then as well as in the definition of shift map $\ShB$, see~\eqref{equ:glob_shift_map}, $\hShift$ is defined only on a certain subset of $\FuncMz$.
Nevertheless, the following lemma shows that if $z$ is a singular point of $\BFld$, then $\hShift$ is well-defined and its image is contained in $\DiffMz$.
\begin{lemma}\label{lm:shift-types}
Suppose that $\BFld(z)=0$.
For $\difM\in \EndMz$ and $\afunc,\bfunc\in \FuncMz$ define the following two maps $\Bsh{\afunc},\hBshh{\bfunc}$ by
$$
\Bsh{\afunc}(x) = \BFlow(x,\afunc(x)),
\qquad
\hBshh{\bfunc}(x) = \BFlow(\difM(x),\bfunc(x)).
$$
Then the following statements hold true.
\begin{enumerate}
 \item[\rm(a)] 
The germs at $z$ of $\Bsh{\afunc}$ and $\hBshh{\bfunc}$ are well-defined.
 \item[\rm(b)]
$\Bsh{\afunc},\Bsh{\bfunc}\in\gimShB$ are germs of a diffeomorphisms at $z$ and 
\begin{equation}\label{equ:Bsh_hBsh_0}
\Bsh{\afunc}^{-1} = \Bsh{-\afunc\circ \Bsh{\afunc}^{-1}}\,, 
\qquad\qquad
\Bsh{\afunc}\circ\Bsh{\bfunc} = \Bsh{\afunc \circ \Bsh{\bfunc} + \bfunc}\,.
\end{equation}
\item[\rm(c)]
If $\difM$ is a germ of a diffeomorphism, then so is $\hBshh{\bfunc}$ and 
\begin{equation}\label{equ:Bsh_hBsh_1}
\hBshh{\bfunc} = \Bsh{\bfunc\circ\difM^{-1}}\circ \difM.
\end{equation}
 \item[\rm(d)]
The following conditions are equivalent:
\begin{equation}\label{equ:Bsh_hBsh_4}
\difM = \Bsh{\afunc} \qquad\qquad \Leftrightarrow \qquad\qquad \hBshh{-\afunc} = \id.
\end{equation}
\end{enumerate}
\end{lemma}
\begin{proof}
(a)~Since $\BFld(z)=0$, it follows from the standard result on dependence of solutions of ODE on initial values that for arbitrary large $A\geq0$ there exists a neighbourhood $W_{A}$ of $z$ such that $W_{A}\times[-A,A] \subset \BDom$.
Hence if $A>|\afunc(z)|$ then $\Bsh{\afunc}$ is defined on some neighbourhood of $z$ contained in $W_{A}$.
The proof for $\hBshh{\bfunc}$ is similar.

(b) is proved in~\cite[Eqs.(8),(9) and Corollary~21]{Maks:TA:2003}, see also~\ref{subsect:h_dif_at_fix}.

(c) Eq.~\eqref{equ:Bsh_hBsh_1} just means that $\hBshh{\bfunc}(x) = \BFlow(\difM(x),\bfunc\circ\difM^{-1}\circ\difM(x))$.

(d) Finally, the verification of~\eqref{equ:Bsh_hBsh_4} is direct. 
\end{proof}

Let $\BFld(z)=0$.
Then we have a well-defined \emph{local shift map at $z$}
$$
\hShift:\FuncMz\to\DiffMz,
\qquad 
\hShift(\afunc) = \BFlow_{\afunc}.
$$
Denote its image $\hShift(\FuncMz)$ in $\DiffMz$ by $\gimShBz$.

Let also $\gEBz$ be the subset of $\DiffMz$ consisting of orbit preserving germs, i.e.\! $\difM\in\DiffMz$ provided there exists a neighbourhood $\Vman$ of $z$ such that $\difM(\orb\cap\Vman)\subset \orb$ for every orbit $\orb$ of $\BFld$.

For every $r=0,1,\ldots,\infty$ let $\gEidBzr{r}$ be the ``identity component'' of $\gEBz$ with respect to $C^{r}_{W}$-topology, i.e.\! $\gEidBzr{r}$ consists of all $\difM\in\gEBz$ for which there exists a neighbourhood $\Vman\subset\Mman$ of $z$ and an $r$-isotopy $H:\Vman\times I\to\Mman$ such that $H_0=i_{\Vman}:\Vman\subset\Mman$, $H_t\in\gEBz$ for all $t\in I$, and $H_1=\difM$.

Then similarly to~\eqref{equ:inclusions_sh_concmp} we have the following inclusions:
$$ 
\gimShBz \subset \gEidBzr{\infty} \subset\cdots\subset \gEidBzr{1} \subset \gEidBzr{0}.
$$

\section{Parameter rigidity}\label{sect:param_rigidity}
In recent years there were obtained many results concerning rigidity of hyperbolic and locally free actions of certain Lie groups and their lattices, see e.g.~\cite{KatokSpatzier:PMIHES:1994, Hurder:CM:1994, Kanai:GAFA:1996, KatokSpatzier:TMIS:1997, MatsumotoMitsumatsu:ETDS:2003, Dajmanovic:LMD:2007, EinsiedlerFisher:IJM:2007} and references there.
Roughly speaking a rigidity of an action $T$ means that every action $T'$ which is sufficiently close in a proper sense to $T$ is conjugate to $T$.

For instance, in a recent paper~\cite{Santos:ETDS:2007} by N.\;dos Santos parameter rigidity of locally free actions of contractible Lie groups on closed manifolds are considered.
In the case of vector fields, i.e.\! actions of $\RRR$, local freeness means regularity of orbits.
In contrast we will consider certain classes of vector fields with singular points, i.e.\! not locally free $\RRR$-actions, and prove their parameter rigidity, see~\ref{th:param-rigid} and \ref{cor:LinVF}.

\begin{definition}\label{defn:param-rigid}{\rm(c.f.~\cite{Santos:ETDS:2007})}
We say that a vector field $\BFld$ on  a manifold $\Mman$ is {\bfseries parameter rigid} if for any vector field $\AFld$ on $\Mman$ such that every orbit of $\AFld$ is contained in some orbit of $\BFld$ there exists a $C^{\infty}$-function $\afunc$ such that  $\AFld=\afunc \BFld$.
\end{definition}

Let $\singB$ and $\singA$ be the sets of singular points of $\BFld$ and $\AFld$ respectively.
The assumption that orbits of $\AFld$ are contained in orbits of $\BFld$ implies that $\singB\subset\singA$ and that $\BFld$ and $\AFld$ are parallel on $\Mman\setminus\singA$.
Therefore there exists a $C^{\infty}$ function $\mu:\Mman\setminus\singB\to\RRR\setminus\{0\}$ such that $\AFld=\mu \BFld$ on $\Mman\setminus\singB$.
Then Definition~\ref{defn:param-rigid} requires that $\mu$ smoothly extends to all of $\Mman$ for any such $\AFld$.

\begin{lemma}[Extensions of shift functions under homotopies]\label{lm:shift-func-on-reg}
Let $\Vman \subset \Mman$ be an open subset, $\afunc_0\in\funcBV$, and $\Hmap:\Vman\times I\to\Mman$ be a $C^{\infty}$-homotopy such that $\Hmap_0=\Shift(\afunc_0)$ and $\Hmap_t\in\imShBV$ for every $t\in I$.
Then there exists a unique $C^{\infty}$ function $\SCT:(\Vman\setminus\singB)\times I\to\RRR$ such that 
\begin{equation}\label{equ:cond_sh_func}
\SCT(x,0)=\afunc_0(x), \qquad\qquad
\Hmap(x,t)=\BFlow(x,\SCT(x,t)),
\end{equation}
for all $(x,t)\in(\Vman\setminus\singB)\times I$.
Thus $\SCT$ is a shift function for $\Hmap$ on $(\Vman\setminus\singB)\times I$ which extends $\afunc_0$.
\end{lemma}
\begin{proof}
This statement was actually established during the proof of \cite[Th.~25]{Maks:TA:2003} for the case $\afunc_0\equiv0$.
But the same arguments show that the proof holds for any $\afunc_0\in\funcBV$.
We leave the details for the reader.
\end{proof}

\begin{definition}\label{defn:lift-smooth-paths}
Let $\Vman \subset \Mman$ be an open subset such that $\singB$ is nowhere dense in $\Vman$.
We will say that the shift map $\ShBV$ of $\BFld$ satisfies {\bfseries smooth path-lifting condition} if for every $C^{\infty}$-homotopy $\Hmap:\Vman\times I\to\Mman$ and $\afunc_0\in\funcBV$ such that $\Hmap_0=\Shift(\afunc_0)$ and $\Hmap_t\in\imShBV$, $(t\in I)$ the shift function $\SCT:(\Vman\setminus\singB)\times I\to\RRR$ of $\Hmap$ satisfying~\eqref{equ:cond_sh_func} smoothly extends to all of $\Vman\times I$.
\end{definition}

See also~\cite{Schwarz:PMIHES:1980}, where the problem of lifting smooth homotopies of orbit spaces of Lie groups is considered.

\begin{theorem}\label{th:param-rigid}
Let $\BFld$ be a vector field on a manifold $\Mman$.
Suppose that for every singular point $z\in\singB$ there exists an open neighbourhood $\Vman$ such that $\imShBV = \EidBVr{\infty}$ and
the corresponding shift map $\ShBV$ satisfies smooth path-lifting condition.
Then $\BFld$ is parameter rigid.
\end{theorem}
\begin{proof}
Let $\AFld$ be a vector field on $\Mman$ such that every orbit of $\AFld$ is contained in some orbit of $\BFld$.
Then there exists a smooth function $\mu:\Mman\setminus\singB\to\RRR$ such that $\AFld=\mu\BFld$.
We have to show that $\mu$ smoothly extends to all of $\Mman$.

We can assume that $\AFld$ generates a global flow $\AFlow:\Mman\times\RRR\to\Mman$.
Otherwise, there exists a smooth function $\bfunc:\Mman\to(0,\infty)$ such that the vector field $\AFld'=\bfunc\AFld$ generates a global flow, see e.g.~\cite[Corollary~2]{Hart:Top:1983}.
Then $\AFld$ and $\AFld'$ have the same orbit foliation.
Moreover, if $\AFld'=\cfunc\BFld$ for some smooth function $\cfunc:\Mman\to\RRR$, then $\AFld=\tfrac{\cfunc}{\bfunc}\,\BFld$, where $\tfrac{\cfunc}{\bfunc}$ is smooth on all of $\Mman$ as well.

Let $z\in\singB$ and $\Vman$ be a neighbourhood of $z$ such that $\ShBV$ satisfies smooth path-lifting condition.
Then for each $t\in\RRR$ we have a well-defined embedding $\AFlow_t|_{\Vman}:\Vman\to\Mman$ belonging to $\EBV$.
Moreover, since $\AFlow_0=\id_{\Mman}=\Shift(0)$ and $\AFlow$ is $C^{\infty}$, it follows that
$$\AFlow_t|_{\Vman} \in \EidBVr{\infty}=\imShBV.$$
Then by smooth path-lifting condition for $\ShBV$ there exists a smooth function $\bar\mu:\Vman\times\RRR$ such that 
\begin{equation}\label{equ:A_Bs}
 \AFlow(x,t) = \BFlow(x,\bar\mu(x,t))
\end{equation}
and $\bar\mu(x,0)=0$ for all $x\in\Vman$.
Let us differentiate both parts of~\eqref{equ:A_Bs} in $t$ and set $t=0$.
Then we will get
$$
\AFld(x) = \frac{\partial \AFlow}{\partial t}(x,0) = 
\frac{\partial \BFlow}{\partial t}(x,\bar\mu(x,0)) \cdot \bar\mu'_{t}(x,0) = 
\BFld(x)\cdot \bar\mu'_{t}(x,0).
$$
Hence $\mu \equiv \bar\mu'_{t}(x,0)$.
Since $\singB$ is nowhere dense in $\Vman$, we obtain that $\mu$ smoothly extends to all of $\Vman$.
Applying these arguments to all $z\in\singB$ we will get that $\mu$ is smooth on all of $\Mman$.
\end{proof}

As an application of this theorem and results of~\cite{Maks:TA:2003} we will now obtain parameter rigidity of linear vector fields and their regular extensions.

\begin{definition}
Let $\Mman, \Nman$ be two manifolds, $\AFld:\Mman\to T\Mman$ be a vector field of $\Mman$ and 
$$\BFld:\Mman\times\Nman\to T(\Mman\times\Nman) = T\Mman\times T\Nman$$ 
be a vector field of $\Mman\times\Nman$ regarded as sections of the corresponding tangent bundles.
Say that $\BFld$ is a {\bfseries regular extension} of $\AFld$ provided that
$$
\BFld(x,y) = (\AFld(x), H(x,y)), \qquad (x,y)\in \Mman\times \Nman,
$$
for some smooth map $H:\Mman\times\Nman\to T\Nman$ such that $H(x,y) \in T_{y}\Nman$.
In other words the ``first'' coordinate function of $\BFld$ ``coincides with $\AFld$'' and does not depend on $y\in\Nman$.
\end{definition}

For instance if $\AFld_i$ is a vector field on a manifold $\Mman_i$, $(i=1,2)$, then their product $\BFld(x,y)=(\AFld_1(x),\AFld_2(y))$ on $\Mman_1\times\Mman_2$ is a regular extension of either of $\AFld_i$.
Every linear vector field $\BFld(x)=Ax$ on $\RRR^n$ is a product of linear vector fields generated by Jordan cells of real Jordan form of $A$.
Moreover, every Jordan cell vector field is a regular extension of a linear vector fields defined by the one of the following matrices:
$(\lambda)$, $\left(\begin{smallmatrix} 0 & 0 \\ 1 & 0 \end{smallmatrix}\right)$,
$\left(\begin{smallmatrix} a & -b \\ b & a \end{smallmatrix}\right)$, where either $\lambda\not=0$ or $b\not=0$, see also~\cite{Venti:JDE:1966}.

\begin{corollary}\label{cor:LinVF}
Let $\BFld$ be a vector field on a manifold $\Mman$ and $\Vman\subset\Mman$ be an open subset.
Suppose that {\bfseries the restriction of $\BFld$ to $\Vman$ is a regular extension of some non-zero linear vector field}.
This means that there exist a non-zero $(m\times m)$-matrix $A$, a smooth manifold $\Nman$, and a diffeomorphism
$\eta:\Vman\to\RRR^{m}\times\Nman$
such that the induced vector field $\eta^{*}\BFld$ on $\RRR^{k}\times\Nman$ is a regular extension of the linear vector field $\AFld(y)=Ay$ on $\RRR^{m}$.
Then $\imShBV=\EidBVr{0}$ and the shift map $\ShBV$ satisfies smooth path-lifting condition.

Hence if every $z\in\singB$ has a neighbourhood $\Vman$ with the above property, then $\BFld$ is parameter rigid.
\end{corollary}
\begin{proof}
The proof follows from~\cite[Theorem~25 \& Theorem~27, statement about (E)-point]{Maks:TA:2003}.
\end{proof}

In~\cite{Siegel:1952,Sternberg:AJM:1957,Venti:JDE:1966} and others there were obtained sufficient conditions for a vector field $\BFld$ defined in a neighbourhood of $\orig\in\RRR^n$ to be linear in some local coordinates at $\orig$.
These results together with Corollary~\ref{cor:LinVF} imply parameter rigidity for a large class of ``hyperbolic'' flows which is in the spirit of mentioned above results of~\cite{KatokSpatzier:PMIHES:1994, Hurder:CM:1994, Kanai:GAFA:1996, KatokSpatzier:TMIS:1997, MatsumotoMitsumatsu:ETDS:2003,Dajmanovic:LMD:2007, EinsiedlerFisher:IJM:2007, Santos:ETDS:2007} concerning rigidity of hyperbolic locally free actions of Lie groups.
A new feature of Corollary~\ref{cor:LinVF} is that $\BFld$ has singularity, and therefore the corresponding $\RRR$-action (i.e. the flow of $\BFld$) is not locally free.

The following lemma gives sufficient condition for a vector field to satisfy smooth path-lifting condition.

\begin{lemma}\label{lm:psm_splc}
Suppose that $\singB$ is nowhere dense in $\Vman$ and for every $\difM\in\imShBV$ there exists a $C^{\infty}_{W}$-neighbourhood $\NN$ in $\imShBV$ and a preserving smoothness map (not necessarily continuous in any sense)
$$
\RS:
\imShBV \;\supset\; \NN \longrightarrow 
\funcBV \;\subset\;C^{\infty}(\Vman,\RRR),
$$
such that $\gdifM(x)=\BFlow(x,\RS(\gdifM)(x))$ for all $\gdifM\in\NN$.
In other words, $\RS$ is a {\bfseries section\/} of $\ShBV$, i.e. $\ShBV\circ\RS=\id(\NN)$.
Then $\ShBV$ satisfies smooth path-lifting condition.
\end{lemma}
\begin{proof}
Before proving this lemma let us make two remarks.

{\bf R1.}~Let $\difM\in\imShBV$ and $\RS_1,\RS_2:\NN\to\funcBV$ be two preserving smoothness sections of $\ShBV$ defined on some $C^{\infty}_{W}$-neighbourhood $\NN$ of $\difM$, and $\Hmap:\Vman\times I \to \Mman$ be a $C^{\infty}$ map such that $\Hmap_0=\difM$ and $\Hmap_t\in\NN$ for all $t\in I$.
Since $\RS_i$ preserves smoothness we have that the following function $\SCT_i:\Vman\times I \to \RRR$ given by
$$
\SCT_i(x,t)=\RS_i(\Hmap_t)(x)
$$
is $C^{\infty}$ as well as $\Hmap$.
\emph{If $\RS_1(\difM)=\RS_2(\difM)$, i.e. $\SCT_1(\cdot,0)=\SCT_2(\cdot,0)$, then $\SCT_1\equiv\SCT_2$ on all of $\Vman\times I$.}

Indeed, since $\SCT_1(\cdot,0)=\SCT_2(\cdot,0)$, it follows from Lemma~\ref{lm:shift-func-on-reg} that $\SCT_1(x,t)=\SCT_2(x,t)$ for all $(x,t)\in(\Vman\setminus\singB)\times I$.
But $\singB$ is nowhere dense in $\Vman$ and each $\SCT_i$ is continuous.
Therefore $\SCT_1=\SCT_2$ on all of $\Vman\times I$.

{\bf R2.}~\emph{Let $\RS:\NN\to\funcBV$ be a preserving smoothness section of $\ShBV$ defined on some neighbourhood $\NN$ of $\difM$ and $\afunc\in\ShBV^{-1}(\difM)$ be any shift function for $\difM$.
Notice that $\RS(\difM)\in\ShBV^{-1}(\difM)$ as well.
Then there exists (possibly) another preserving smoothness section $\RS':\NN\to\funcBV$ such that $\RS'(\difM)=\afunc$.}

Suppose that $\RS(\difM)\not=\afunc$.
Then $\ShBV$ is not injective map.
Since $\singB$ is nowhere dense, it follows from \cite[Lm.~5 \& Th.~12(2)]{Maks:TA:2003} that there exists a smooth function $\nu:\Vman\to(0,+\infty)$ such that $\BFlow(x,\nu(x))\equiv x$ for all $x\in\Vman$, and 
$\afunc=\RS(\difM)+n\nu$ for some $n\in\ZZZ$.
Define the following map $\RS':\NN\to\funcBV$ by $\RS'(\gdifM)=\RS(\gdifM)+ n\nu$ for $\gdifM\in\NN$.
Then $\RS'$ is also a preserving smoothness section of $\ShBV$ and $\RS'(\difM)=\RS(\difM)+ n\nu=\afunc$.

\medskip 

Now we are ready to complete Lemma~\ref{lm:psm_splc}.
Let $\afunc_0\in\funcBV$ and $\Hmap:\Vman\times I\to\Mman$ be a $C^{\infty}$ map such that $\Hmap_0=\Shift(\afunc_0)$ and $\Hmap_t\in\imShBV$ for all $t\in I$, i.e. $\Hmap_t=\ShBV(\afunc_t)$ for some (not necessarily unique) $\afunc_t\in\funcBV$.

We will show that under assumptions of lemma it is possible to choose $\afunc_t$ so that the correspondence $(x,t)\mapsto\afunc_t(x)$ becomes a $C^{\infty}$ shift function $\SCT':\Vman\times I\to\RRR$ for $\Hmap$ such that $\SCT'(x,0)=\afunc_0(x)$.
Let also $\SCT:(\Vman\setminus \singB)\times I\to \RRR$ be a unique $C^{\infty}$ shift function for $\Hmap$ such that $\SCT(x,0)=\afunc_0(x)=\SCT'(x,0)$, see Lemma~\ref{lm:shift-func-on-reg}.
Then it will follow from uniqueness of such shift function that $\SCT=\SCT'$ on $(\Vman\setminus\singB)\times I$.
Since $\singB$ is nowhere dense in $\Vman$, we will get $\SCT\equiv\SCT'$ on all of $\Vman\times I$.
This will imply smooth path-lifting condition for $\ShBV$.

Notice that $\Hmap$ can be regarded as a continuous path into $C^{\infty}_{W}$ topology of $\imShBV$.
Since the image of this path is compact, there exist finitely many points $0=t_0<t_1<\cdots<t_n=1$ and for each $k=0,1,\ldots,n-1$ a $C^{\infty}$ neighbourhood $\NN_{k}$ of $\Hmap_{t_k}$ in $\imShBV$ such that 
\begin{itemize}
\item
$\Hmap_{t}\in\NN_{k}$ for all $t\in[t_{k},t_{k+1}]$, 
\item
and for every $\afunc\in\funcBV$ such that $\ShBV(\afunc)=\Hmap_{t_k}$ there exists a preserving smoothness section $\RS_{\afunc}:\NN_{k}\to\funcBV$ of $\ShBV$ such that $\RS_{\afunc}(\Hmap_{t_k})=\afunc$, (this follows from {\bf R2}).
\end{itemize}

Since $\Hmap_0=\ShBV(\afunc_0)$, the map $\RS_{\afunc_0}$ is well defined and we put 
$$
\SCT(x,t) = \RS_{\afunc_0}(\Hmap_t)(x),
\qquad (x,t)\in\Vman\times[0,t_1].
$$
Then $\SCT$ is smooth on $\Vman\times[0,t_1]$.

Denote $\afunc_{t_1}=\RS_{\afunc_0}(\Hmap_{t_1})$.
Then $\Hmap_{t_1}=\ShBV(\afunc_{t_1})$ and $\RS_{\afunc_1}$ is well defined.
Therefore we also put
$$
\SCT(x,t) = \RS_{\afunc_1}(\Hmap_t)(x),
\qquad (x,t)\in\Vman\times[t_1,t_2].
$$
Then $\SCT$ is smooth on $\Vman\times[t_1,t_2]$.
Moreover 
$\RS_{\afunc_0}$ and $\RS_{\afunc_1}$ are two sections of $\ShBV$ such that 
$\RS_{\afunc_0}(\Hmap_{t_1})=\RS_{\afunc_1}(\Hmap_{t_1})=\afunc_{t_1}$.
Therefore by {\bf R1} $\RS_{\afunc_0}(\Hmap_{t})=\RS_{\afunc_1}(\Hmap_{t}) = \SCT(\cdot,t)$ for all $t$ sufficiently close to $t_1$.
Since $\RS_{\afunc_0}$ and $\RS_{\afunc_1}$ preserve smoothness, it follows that $\SCT$ is smooth on all of $[0,t_2]$.

Using induction on $n$ we can smoothly extend $\SCT$ on all of $\Vman\times I$.
\end{proof}

\section{$\infty$-jets of shifts}\label{sect:charact_inf_jets}
Let $\BFld$ be a smooth vector field near the origin $\orig\in\RRR^n$ such that $\BFld(\orig)=0$ and  $(\BFlow_{t})$ be the local flow of $\BFld$.
Define the following map
$$j^{\infty}:\DiffRno\to (\RRR[[x_1,\ldots,x_n]])^{n}$$ associating to every $\difM\in\DiffRno$ its $\infty$-jet at $\orig$.
Let also 
\begin{equation}\label{equ:gJShBO}
\gJShBO = (\jo{\infty})^{-1}\left[ \jo{\infty}\left(\,\gimShBO\, \right) \right].
\end{equation}
Thus $\gJShBO$ is the subgroup of $\DiffRno$ consisting of germs $\difM$ for which there exists a smooth function $\afunc_{\difM}\in\FuncRno$ such that $\jo{\infty}(\difM)=\jo{\infty}(\Bsh{\afunc_{\difM}})$.

Evidently, $\gimShBO\subset\gJShBO$.
Our first result gives necessary and sufficient conditions for a subgroup $\grp\subset\DiffRno$ containing $\gimShBO$ to be included into $\gJShBO$.

\begin{theorem}\label{th:grp_subset_JSh}
Suppose that $\BFld$ is not flat at $\orig$, i.e.\! there exists $\dg\geq1$ such that $\jo{\dg-1}(\BFld)=0$ and 
$\PFunc = \jo{\dg}(\BFld):\RRR^n\to\RRR^n$
is a non-zero homogeneous map of degree $\dg$.
For $\dg=1$ we will write $\PFunc(x) = \Lmatr\cdot x$, where $\Lmatr$ is a certain non-zero $(n\times n)$-matrix.

Let $\grp$ be a subgroup of $\DiffRno$ having the following properties:
\begin{large}\end{large}\begin{enumerate}
 \item[\coA] 
$\gimShBO \subset \grp$.
 \item[\coB]
For every $\difM\in\grp$ there exists $\omega_0\in\RRR$ such that 
$$\jo{p}(\difM)(x) = \jo{p}(\Bsh{\omega_0})(x)=
\left\{
\begin{array}{ll}
e^{\Lmatr\, \omega_0} \cdot x, & \dg=1,  \\ [1.5mm]
x + \PFunc(x)\cdot \omega_0, & \dg\geq2.
\end{array}
\right.$$
 \item[\coC]
Moreover, if $\jo{k-1}(\difM) = \jo{k-1}(\id)$ for some $k\geq\dg$, then there exists a unique homogeneous polynomial $\omega_{l}$ of degree $l=k-\dg$ such that 
$$\jo{k}(\difM)(x) = \jo{k}(\Bsh{\omega_l})(x) = x + \PFunc(x) \cdot \omega_{l}(x).$$
\end{enumerate}
Then $\jo{\infty}\bigl(\gimShBO\bigr)=\jo{\infty}(\grp)$.
In other words, $\gimShBO \subset \grp \subset \gJShBO$.
\end{theorem}

The rest of this section is devoted to the proof of Theorem~\ref{th:grp_subset_JSh} which will be completed in~\S\ref{sect:th-main-equiv}.
Our aim is to establish Lemma~\ref{lm:init-terms-shifts} and statement (2) of Corollary~\ref{cor:rel_betw_shifts} below.
They will be used in the proof of Theorem~\ref{th:grp_subset_JSh}.

\subsection{Spaces of jets}\label{sect:jets}
Let $\EndRnoRm$ be the space of germs at the origin $\orig\in\RRR^n$ of smooth maps $\difM:\RRR^n\to\RRR^m$ and $\EndRnoRmo$ be its subset consisting of all germs such that $\difM(\orig)=\orig$.
For $n=m$, we will write $\EndRno$ instead of $\EndRnoRno$.
Let also $\DiffRno \subset \EndRno$ be the subset consisting of germs of diffeomorphisms at $\orig$.
The space $\EndRnoR$ of germs of smooth functions will be denoted by $\FuncRno$.

For $\difM\in\EndRnoRm$ denote by $\jo{k}(\difM)$ its $k$-jet at $\orig\in\RRR^n$.
It will be convenient to formally assume that $(-1)$-jet of $\difM$ is identically zero: 
\begin{equation}\label{equ:jet-1}
\jo{(-1)}(\difM)\equiv 0.
\end{equation}
We will say that $\difM \in \EndRnoRm$ is \emph{$k$-small}, $(k\geq0)$, at $\orig$ provided $\jo{k-1}(\difM)=0$.
In particular, by assumption~\eqref{equ:jet-1} every $\difM\in\EndRnoRm$ is $0$-small and $\difM$ is $1$-small iff $\difM\in\EndRnoRmo$, i.e. $\jo{0}(\difM)=\difM(\orig)=\orig$.
If $\jo{\infty}(\difM)=0$ then $\difM$ is called \emph{flat}.

Let $\afunc,\bfunc\in\FuncRno$  be such that $\afunc$ is $a$-small, and $\bfunc$ is $b$-small for some $a,b\geq0$.
Then their product $\afunc\bfunc$ is $(a+b)$-small.
In other words
\begin{equation}\label{equ:jet_prod}
\jo{a-1}(\afunc)=\jo{b-1}(\bfunc)=0
\qquad \Rightarrow\qquad
\jo{a+b-1}(\afunc\bfunc)=0.
\end{equation}

We also say that $\difM\in\EndRnoRmo$ is \emph{homogeneous of degree $k$} if its coordinate functions are homogeneous polynomials of degree $k$.
Let $\JRnok$, $(0\leq k < \infty)$, be the space of all polynomial maps $\difM:\RRR^n\to\RRR^n$ of degree $\leq k$ such that $\difM(\orig)=\orig$.
Similarly, put
$$\JRnoi = \bigl\{ \tau \in \bigl(\RRR[[x_1,\ldots,x_n]] \bigr)^n \ : \ \tau(\orig)=0 \bigr\}.$$

Define the following linear map $\jo{k}:\EndRno \to \JRnok$ associating to every $\difM\in\EndRno$ its $k$-jet $\jo{k}(\difM)$ at $\orig$.
Then it is easy to verify that:
\begin{equation}\label{equ:jfg_jf_jg}
 \jo{k}(f\circ g) = \jo{k}(\jo{k}(f) \circ \jo{k}(g)),
\qquad
f,g\in\EndRno,
\end{equation}
see e.g.~\cite[\S4]{Sternberg:AJM:1957}.
The following lemma is a direct corollary of~\eqref{equ:jfg_jf_jg}.
\begin{lemma}\label{lm:rel_between_jets}
Let $f,g \in \DiffRno$.
Then for every $k=0,\ldots,\infty$ the following conditions are equivalent:
$$
\jo{k}(f)=\jo{k}(g) 
\quad \Leftrightarrow\quad 
\jo{k}(f^{-1}\circ g)=\jo{k}(\id)
\quad \Leftrightarrow\quad 
\jo{k}(f^{-1})=\jo{k}(g^{-1}).
$$ 
Moreover if $\afunc,\bfunc\in\FuncRno$, then
$$
\jo{k}(\afunc)=\jo{k}(\bfunc)
\quad \Longleftrightarrow\quad 
\jo{k}(\afunc\circ f)=\jo{k}(\bfunc \circ f). \qed
$$
\end{lemma}

\subsection{Jets of a flow.}
First we introduce the following notation.
For a smooth mapping $\AFld=(\AFld^1,\ldots,\AFld^n):\RRR^n\to\RRR^n$ denote
$$
\nabla\AFld =
\left\|
\begin{array}{cccc}
\frac{\partial \AFld^1}{\partial x_1} & \frac{\partial \AFld^1}{\partial x_2} & \cdots & \frac{\partial \AFld^1}{\partial x_n} \\ [2mm]
\frac{\partial \AFld^2}{\partial x_1} & \frac{\partial \AFld^2}{\partial x_2} & \cdots & \frac{\partial \AFld^2}{\partial x_n} \\ [2mm]
\cdots & \cdots & \cdots & \cdots  \\ [2mm]
\frac{\partial \AFld^n}{\partial x_1} & \frac{\partial \AFld^n}{\partial x_2} & \cdots & \frac{\partial \AFld^n}{\partial x_n}
\end{array}
\right\|.
$$
Thus $\nabla\AFld$ is an $(n\times n)$-matrix whose rows are the gradients of the corresponding coordinate functions of $\AFld$.

Now let $\BFld=(\BFld^1, \ldots,\BFld^n)$ be a smooth vector field defined on some neighbourhood $\Vman$ of $\orig\in\RRR^n$ and 
$$\BFlow: \Vman\times\RRR \supset \;\BDom\;\longrightarrow\;\RRR^n$$ be the local flow of $\BFld$, so
\begin{equation}\label{equ:def-flow-ODE}
\frac{\partial\BFlow}{\partial t}(x,t) = \BFld(\BFlow(x,t)) \qquad \text{and} \qquad \BFlow(x,0)=x.
\end{equation}
Hence the Taylor expansion of $\BFlow$ in $t$ at $x=\orig$ is given by
\begin{equation}\label{equ:BFlow_t}
\BFlow(x,t) = x + \vfunc_1(x) \;t + \vfunc_2(x)\; \frac{t^2}{2} + \ldots + \vfunc_{n}(x)\; \frac{t^{n}}{n!} + \cdots,
\end{equation}
where $\vfunc_i(x)  = \frac{\partial^i\BFlow}{\partial t^i}(x,t)|_{t=0}$.
It follows from~\eqref{equ:def-flow-ODE} that $\vfunc_1=\BFld$.
\begin{lemma}\label{lm:ODEs-for-flow}
For every $i\geq1$ we have that
$$
\vfunc_{i+1}=\nabla\vfunc_i \cdot \BFld =
\left\|
\begin{array}{cccc} 
\frac{\partial \vfunc_i^1}{\partial x_1} & \frac{\partial \vfunc_i^1}{\partial x_2} & \cdots & \frac{\partial \vfunc_i^1}{\partial x_n} \\ [2mm]
\frac{\partial \vfunc_i^2}{\partial x_1} & \frac{\partial \vfunc_i^2}{\partial x_2} & \cdots & \frac{\partial \vfunc_i^2}{\partial x_n} \\ [2mm]
\cdots & \cdots & \cdots & \cdots  \\ [2mm]
\frac{\partial \vfunc_i^n}{\partial x_1} & \frac{\partial \vfunc_i^n}{\partial x_2} & \cdots & \frac{\partial \vfunc_i^n}{\partial x_n}
\end{array}
\right\|
\cdot
\left\|
\begin{array}{c}
\BFld^1 \\ [2mm]
\BFld^2 \\ [2mm]
\cdots \\ [2mm]
\BFld^n
\end{array}
\right\|
=
\left\|
\begin{array}{c}
\langle \nabla \vfunc_i^1,\BFld\rangle  \\ [2mm]
\langle \nabla \vfunc_i^2,\BFld\rangle  \\ [2mm]
\cdots \\ [2mm]
\langle \nabla \vfunc_i^n,\BFld\rangle
\end{array}
\right\|,
$$
where $\vfunc_i^j$ is the $j$-th coordinate function of $\vfunc_i$.
Moreover,
$\frac{\partial^i\BFlow}{\partial t^i}(x,t) = \vfunc_i (\BFlow(x,t)).$
If $\jo{\dg-1}(\BFld)=0$, then 
\begin{equation}\label{equ:estim_ord_vi}
\jo{i(\dg-1)}(\vfunc_i)=0, \qquad \forall i\geq1.
\end{equation}
\end{lemma}
\begin{proof}
We will now calculate $\vfunc_2$.
Let $\BFlow^j$ be the $j$-th coordinate function of $\BFlow$.
Then 
\begin{align*}
 \frac{\partial^2\BFlow^j}{\partial t^2}(x,t)  & =
\frac{\partial}{\partial t}\frac{\partial\BFlow^j}{\partial t}(x,t) \overset{\eqref{equ:def-flow-ODE}}{=\!=\!=}
\frac{\partial}{\partial t}\BFld^j(\BFlow(x,t)) \overset{\eqref{equ:def-flow-ODE}}{=\!=\!=} \\
&  =   \sum\limits_{k=1}^{n} \frac{\partial \BFld^j}{\partial x_k} (\BFlow(x,t)) \cdot \BFld^k(\BFlow(x,t)) = \\
& =   \Bigl( \sum\limits_{j=1}^{n} \frac{\partial \BFld^j}{\partial x_k} \cdot \BFld^k \Bigr) \circ \BFlow(x,t) 
 = \langle \grad\BFld^j, \BFld \rangle \circ \BFlow(x,t).
\end{align*}
Therefore
$$\frac{\partial^2\BFlow}{\partial t^2}(x,t) = (\nabla\BFld \cdot \BFld) \circ \BFlow(x,t) = \vfunc_2 \circ \BFlow(x,t).$$

Calculations for other $\vfunc_i$ are similar and are left to the reader.

\noindent{\em Proof of \eqref{equ:estim_ord_vi}.} 
We have $\jo{1\cdot (\dg-1)}(\vfunc_1) = \jo{\dg-1}(\BFld)=0$.
Suppose by induction that $\jo{i(\dg-1)}(\vfunc_{i})=0$ for some $i$.

Then $\jo{i(\dg-1)-1}(\nabla\vfunc_{i})=0$.
Since $\jo{\dg-1}(\BFld)=0$, it follows from~\eqref{equ:jet_prod} that
$\jo{i(\dg-1)+\dg-1}(\nabla\vfunc_{i}\cdot\BFld)=\jo{(i+1)(\dg-1)}(\vfunc_{i+1})=0$.
\end{proof}

\subsection{Initial non-zero jets of smooth shifts for non-flat vector fields.}
Suppose now that there exists $\dg\geq1$ such that $\jo{\dg-1}(\BFld)=0$ and 
$$\PFunc = \jo{\dg}(\BFld):\RRR^n\to\RRR^n$$
is a non-zero homogeneous map of degree $\dg$.
For $\dg=1$ we will write $$\PFunc(x) = \Lmatr\cdot x,$$ where $\Lmatr$ is a certain non-zero $n\times n$-matrix.
\begin{lemma}\label{lm:init-terms-shifts}
Let $\afunc:\RRR^n\to\RRR$ be an $l$-small germ at $\orig$ for some $l\geq0$, so $\jo{l}(\afunc)=\omega$ is a homogeneous polynomial of degree $l$.
Put $\Bsh{\afunc}(x) = \BFlow(x,\afunc(x))$.
Then
\begin{equation}\label{equ:pljet_Fa}
\jo{\dg+l}(\Bsh{\afunc})(x) = \left\{
\begin{array}{ll}
e^{\Lmatr\, \omega} \cdot x, & \dg=1 \ \text{and} \ l=0,  \\ [1.5mm]
x + \PFunc(x)\cdot \omega(x), & \text{otherwise}.
\end{array}
\right.
\end{equation}
Thus the $(\dg+l)$-jet of $\Bsh{\afunc}$ depends only on the $l$-jet of $\afunc$.
Moreover, 
\begin{equation}\label{equ:jFainv_jF_a}
\jo{\dg+l}(\Bsh{\afunc}^{-1})=\jo{\dg+l}(\Bsh{-\afunc}). 
\end{equation}
\end{lemma}
\begin{proof}
\eqref{equ:pljet_Fa}.
Substituting $\afunc$ into~\eqref{equ:BFlow_t} instead of $t$ we get
$$
\Bsh{\afunc}(x)=
\BFlow(x,\afunc(x))=x + \BFld(x) \;\afunc(x) + \cdots + \vfunc_{i}(x)\; \frac{\afunc(x)^{i}}{i!} + \cdots
$$
Though this series converges near $z$, the infinitesimal orders of its summands at $\orig$ do not necessarily increase when $n\to\infty$.
Therefore in order to find the initial non-zero jet of $\Bsh{\afunc}$ we should investigate each of $\vfunc_{i}\afunc^{i}$.
In fact we will get only one exceptional case $\dg=1$ and $l=0$.

First we calculate $\jo{1}(\Bsh{\afunc})$.
Suppose that $\jo{1}(\BFld) = \Lmatr\cdot x$ for some possibly zero matrix $\Lmatr$.
Then $\jo{0}(\nabla\BFld)=\nabla\BFld(\orig)=\Lmatr$, 
$$
\begin{array}{rcl}
\jo{1}(\vfunc_2) = \jo{1}(\nabla\BFld \cdot \BFld) & = & 
\jo{0}(\nabla\BFld) \cdot \jo{1}(\BFld) +
\jo{1}(\nabla\BFld)\cdot \jo{0}(\BFld) \
= \\ [2mm] 
& =  & 
\Lmatr \cdot \Lmatr \cdot x  + \jo{1}(\nabla\BFld)\cdot 0 \ = \ \Lmatr^2\cdot x,
\end{array}
$$
and by induction on $i$ we obtain
$$
\jo{1}(\vfunc_i)= \jo{1}\bigl(\nabla\vfunc_{i-1} \cdot \BFld\bigr) = \Lmatr^i\cdot x.
$$
Therefore 
$$
\begin{array}{rcl}
\jo{1}(\vfunc_{i}\afunc^{i}) & = &
\jo{1}(\vfunc_{i}) \cdot \jo{0}(\afunc^{i}) + 
\jo{0}(\vfunc_{i}) \cdot \jo{1}(\afunc^{i}) = \\ [2mm]
& = & L^i \cdot x \cdot \afunc(\orig) + 0 \cdot \jo{1}(\afunc^{i}) = 
\afunc(\orig) \cdot L^i \cdot  x,
\end{array}
$$
whence 
$$
\jo{1}(\Bsh{\afunc})(x) = 
 x + \sum_{i=1}^{\infty} \frac{\afunc(\orig)^{i}}{i!} \cdot \Lmatr^{i}\cdot x =
e^{\Lmatr \cdot \afunc(\orig)} \cdot x.
$$

Thus if $p=1$ and $l=0$, i.e $\Lmatr\not=0$ and $\jo{0}(\afunc)=\afunc(\orig)=\omega\not=0$, we obtain that $\jo{\dg+l}(\Bsh{\afunc})(x) = \jo{1}(\Bsh{\afunc})(x)= e^{\Lmatr \,\omega} \cdot x$.

Otherwise $\dg+l\geq2$.
We claim in this case $\jo{\dg+l}(\vfunc_i \cdot \afunc^i) = 0$ for $i\geq2$.
This will imply that 
$$
\jo{\dg+l}(\Bsh{\afunc})(x) = 
\jo{\dg+l}(x + \BFld(x) \afunc(x)) = 
x + \PFunc(x)\cdot \omega(x).
$$
To calculate $\jo{\dg+l}(\vfunc_i \cdot \afunc^i)$ notice that by~\eqref{equ:estim_ord_vi} $\jo{i(\dg-1)}(\vfunc_i)=0$ and by assumption $\jo{l-1}(\afunc)=0$ as well.
Then it follows from~\eqref{equ:jet_prod} that $\jo{il-1}(\afunc^i)=0$ and
$$\jo{i(\dg-1) + il}(\vfunc_i \cdot \afunc^i) = \jo{i(\dg+l-1)}(\vfunc_i \cdot \afunc^i) = 0.$$
It remains to note that $\dg+l < i(\dg+l-1)$ if $i\geq2$ and $\dg+l\geq2$.
Hence $\jo{\dg+l}(\vfunc_i \, \afunc^{i}) = 0$ for $i\geq2$.

\emph{Proof of~\eqref{equ:jFainv_jF_a}}.
Recall that by~\eqref{equ:Bsh_hBsh_0} $\Bsh{\afunc}^{-1}=\Bsh{-\afunc\circ \Bsh{\afunc}^{-1}}$. We will show that $\jo{l}(\afunc\circ\Bsh{\afunc}^{-1})=\jo{l}(\afunc)$.
Then it will follow from~\eqref{equ:pljet_Fa} that
$$
\jo{\dg+l}(\Bsh{\afunc}^{-1})=
\jo{\dg+l}(\Bsh{-\afunc\circ \Bsh{\afunc}^{-1}}) \stackrel{\eqref{equ:pljet_Fa}}{=\!=\!=}
\jo{\dg+l}(\Bsh{-\afunc}).
$$

Suppose that $l=0$.
Since $\BFlow_t(\orig)=\orig$ for all $t\in\RRR$, we obtain that $\jo{0}(\afunc\circ \Bsh{\afunc}^{-1})=\jo{0}(\afunc)=\afunc(\orig)$. 

If $l\geq1$, then it follows from~\eqref{equ:pljet_Fa} that $\jo{1}(\Bsh{\afunc})(x)=x$, whence $\jo{l}(\afunc\circ \Bsh{\afunc}^{-1})=\jo{l}(\afunc)$.
\end{proof}

\begin{corollary}\label{cor:rel_betw_shifts}
Let $\afunc,\bfunc\in \FuncRno$, $\difM\in\DiffRno$, $l,k=0,1,\ldots,\infty$.
\begin{enumerate}
\item[\rm(1)]
The following conditions {\rm(A)-(C)} are equivalent:
  \begin{enumerate}
    \item[\rm(A)] 
      $\jo{l}(\afunc)=\jo{l}(\bfunc)$,
    \item[\rm(B)] 
      $\jo{\dg+l}(\Bsh{\afunc})=\jo{\dg+l}(\Bsh{\bfunc})$,
    \item[\rm(C)] 
      $\jo{\dg+l}(\hBshh{\afunc})=\jo{\dg+l}(\hBshh{\bfunc})$.
  \end{enumerate}

\item[\rm(2)]
The following conditions {\rm(D)} and {\rm(E)} are equivalent:
   \begin{enumerate}
     \item[\rm(D)]
       $\jo{k}(\difM) = \jo{k}(\Bsh{\afunc})$,
     \item[\rm(E)]
       $\jo{k}(\hBshh{-\afunc})=\jo{k}(\id)$.
     \end{enumerate}
\end{enumerate}
\end{corollary}
\begin{proof}
(1) (A)$\Leftrightarrow$(B).
Notice that
\begin{equation}\label{equ:Fa_Fb_1}
\Bsh{\afunc}\circ\Bsh{\bfunc}^{-1} \stackrel{\eqref{equ:Bsh_hBsh_0}}{=\!=\!=}
\Bsh{\afunc}\circ\Bsh{-\bfunc \circ \Bsh{\bfunc}^{-1}} 
\stackrel{\eqref{equ:Bsh_hBsh_0}}{=\!=\!=}
\Bsh{\afunc \circ \Bsh{\bfunc}^{-1}-\bfunc \circ \Bsh{\bfunc}^{-1}} = 
\Bsh{(\afunc-\bfunc) \circ \Bsh{\bfunc}^{-1}}.
\end{equation}
Then the following statements are equivalent:
$$
\begin{array}{lcl}
\text{(A)}~\jo{l}(\afunc)=\jo{l}(\bfunc),  & \qquad &
 \text{(c)}~\jo{\dg+l}(\Bsh{(\afunc-\bfunc) \circ \Bsh{\bfunc}^{-1}})=\jo{\dg+l}(\id), \\
\text{(a)}~\jo{l}(\afunc-\bfunc)=0, & &
  \text{(d)}~\jo{\dg+l}(\Bsh{\afunc}\circ\Bsh{\bfunc}^{-1})=\jo{\dg+l}(\id), \\
\text{(b)}~\jo{l}((\afunc-\bfunc) \circ \Bsh{\bfunc}^{-1})=0, & &
\text{(B)}~\jo{\dg+l}(\Bsh{\afunc}) = \jo{\dg+l}(\Bsh{\bfunc}).
\end{array}
$$
The equivalence of (A), (a), and (b) is trivial, (b)$\Leftrightarrow$(c) holds by~\eqref{equ:pljet_Fa},
(c)$\Leftrightarrow$(d) by~\eqref{equ:Fa_Fb_1}, and (d)$\Leftrightarrow$(B) by Lemma~\ref{lm:rel_between_jets}.

(A)$\Leftrightarrow$(C)
Recall that by~\eqref{equ:Bsh_hBsh_1} $\hBshh{\afunc} = \Bsh{\afunc\circ\difM^{-1}} \circ \difM$.
Then the following conditions are equivalent:
$$
\begin{array}{lcl}
\text{(C)}~\jo{\dg+l}(\hBshh{\afunc})=\jo{\dg+l}(\hBshh{\bfunc}),  & \qquad & 
  \text{(f)}~\jo{l}(\afunc\circ\difM^{-1})=\jo{l}(\bfunc\circ\difM^{-1}), \\
\text{(e)}~\jo{\dg+l}(\Bsh{\afunc\circ\difM^{-1}})=\jo{\dg+l}(\Bsh{\bfunc\circ\difM^{-1}}), & & 
\text{(A)}~\jo{l}(\afunc)=\jo{l}(\bfunc).
\end{array}
$$
The equivalence (C)$\Leftrightarrow$(e) holds by~\eqref{equ:Bsh_hBsh_1}, (e)$\Leftrightarrow$(f) by the equivalence (B)$\Leftrightarrow$(A) which is already proved, and (f)$\Leftrightarrow$(A) by Lemma~\ref{lm:rel_between_jets}.

(2) (D)$\Rightarrow$(E).
Suppose that $\jo{k}(\difM) = \jo{k}(\Bsh{\afunc})$.
Then 
\begin{gather}
\label{equ:j1}
\jo{k}(\Bsh{\afunc}^{-1} \circ \difM)=\jo{k}(\id),  \\
\label{equ:j2}
\jo{k}(\afunc)=\jo{k}(\afunc\circ \Bsh{\afunc}^{-1} \circ\difM).
\end{gather}
Notice also that 
\begin{equation}\label{equ:j3}
\Bsh{\afunc}^{-1} \circ \difM 
\stackrel{\eqref{equ:Bsh_hBsh_0}}{=\!=\!=}
\Bsh{-\afunc\circ \Bsh{\afunc}^{-1}}\circ \difM =
\hBshh{-\afunc\circ \Bsh{\afunc}^{-1} \circ\difM}\,.
\end{equation}
Hence 
$$
\jo{k}(\hBshh{-\afunc})
    \stackrel{\eqref{equ:j2}}{=\!=\!=}
\jo{k}(\hBshh{-\afunc\circ \Bsh{\afunc}^{-1} \circ\difM})
    \stackrel{\eqref{equ:j3}}{=\!=\!=}
\jo{k}(\Bsh{\afunc}^{-1} \circ \difM)
    \stackrel{\eqref{equ:j1}}{=\!=\!=}
\jo{k}(\id).
$$

(E)$\Rightarrow$(D).
It is easy to verify that
\begin{equation}\label{equ:j4}
\Bsh{\afunc\circ\hBshh{-\afunc}^{-1}} \circ \hBshh{-\afunc} = \difM.
\end{equation}
This identity simply means that $\BFlow\bigl(\BFlow(\difM(x),-\afunc(x)),\afunc(x)\bigr)=\difM(x)$.
Suppose that $\jo{k}(\hBshh{-\afunc})=\jo{k}(\id)$.
Then 
\begin{equation}\label{equ:j5}
\jo{k}(\afunc\circ\hBshh{-\afunc})=\jo{k}(\afunc), 
\end{equation}
whence
$$
\jo{k}(\difM) 
    \stackrel{\eqref{equ:j4} \,\&\, (E)}{=\!=\!=\!=\!=\!=\!=}
\jo{k}(\Bsh{\afunc\circ\hBshh{-\afunc}^{-1}})
    \stackrel{\eqref{equ:j5} \,\&\, \eqref{equ:pljet_Fa}}{=\!=\!=\!=\!=\!=\!=}
\jo{k}(\Bsh{\afunc}). 
$$
 Corollary~\ref{cor:rel_betw_shifts} is proved.
\end{proof}

\subsection{Proof of Theorem~\ref{th:grp_subset_JSh}}\label{sect:th-main-equiv}
Let $\BFld$ be a vector field defined on some neighbourhood of the origin in $\RRR^{n}$ and  $\grp$ be a subgroup of $\DiffRno$ satisfying \coA-\coC.
We have to show that $\jo{\infty}\bigl(\gimShBO\bigr)=\jo{\infty}(\grp)$.
Due to \coA\ it remains to verify that $\jo{\infty}\bigl(\gimShBO\bigr) \supset\jo{\infty}(\grp)$.

Let $\difM\in\grp$.
We will find a germ of a smooth function $\afunc\in\FuncRno$ such that $\jo{\infty}(\hBshh{-\afunc})=\jo{\infty}(\id)$.
Then it will follow from (2) of Corollary~\ref{cor:rel_betw_shifts} that $\jo{\infty}(\hBsh{\afunc})=\jo{\infty}(\difM)$, i.e. $\jo{\infty}(\difM)\in\jo{\infty}(\gimShBO)$.

Since $\grp$ is a group, it follows from \coA\ and~\eqref{equ:Bsh_hBsh_1} that 
\begin{equation}\label{equ:Fah_in_G}
\hBshh{\afunc} = \Bsh{\afunc\circ\difM^{-1}}\circ \difM \;\in\; \grp,
\qquad \forall \afunc\in\FuncRno. 
\end{equation}
Put $\difM_0=\difM$.
Then by \coB\ there exists $\omega_0\in\RRR$ such that $\jequ{\difM_0}{\Bsh{\omega_0}}{\dg}$.
Denote
$$
\difM_1(x) = \qBsh{\difM}{-\omega_0}(x) = \BFlow(\difM(x),-\omega_0) =\BFlow_{-\omega_0} \circ \difM(x).
$$
Then $\difM_1\in\grp$ and by (2) of Corollary~\ref{cor:rel_betw_shifts} $\jequ{\difM_1}{\id_{\RRR^{n}}}{\dg}.$

Therefore by \coC\ there exists a homogeneous polynomial $\omega_1$ of degree $1$ such that $\jequ{\difM_1}{\Bsh{\omega_1}}{\dg + 1}$, where 
$\Bsh{\omega_1}(x) = \BFlow(x,\omega_1(x))$.
Denote 
$$
\begin{array}{rcl}
\difM_2(x)  =  \qBsh{\difM_1}{-\omega_1}(x) & = & \BFlow(\difM_1(x), -\omega_1(x)) = \\ [1.5mm]
& = & \BFlow(\BFlow(\difM_0(x), -\omega_0), -\omega_1(x))=  \\ [1.5mm]
& = & \BFlow(\difM(x), -\omega_0-\omega_1(x)) = \qBsh{\difM}{-\omega_0-\omega_1}(x).
\end{array}
$$
Then by~\eqref{equ:Fah_in_G} $\difM_2\in\grp$ and by (2) of Corollary~\ref{cor:rel_betw_shifts} $\jequ{\difM_2}{\id}{\dg + 1}$.
Therefore we can again apply \coC\ to $\difM_2$ and so on.
Using induction we will construct a sequence of homogeneous polynomials $\{\omega_l\}_{l=0}^{\infty}$, $(\deg\omega_l=l)$, and a sequence $\{\difM_l\}$ in $\grp$ such that for every $l\geq0$, we have that 
\begin{gather}
\label{equ:jets_of_hl}
\jequ{\difM_{l}}{\id}{\dg+l-1 }, \qquad \qquad
\jequ{\difM_l}{\Bsh{\omega_l}}{\dg + l},
\\
\label{equ:hlplus1}
\difM_{l+1}(x) = 
\qBsh{\difM_l}{-\omega_l}(x)   = 
\BFlow\bigl(\difM(x),-\sum\limits_{i=0}^{l}\omega_i(x)\bigr).
\end{gather}

Put $\tau=\sum\limits_{l=0}^{\infty}\omega_l(x)$.
Then by a well-known theorem of E.~Borel, see also \S\ref{sect:Borel_th}, there exists a smooth function $\afunc:\RRR^n\to\RRR$ whose $\infty$-jet at $\orig$ coincides with $\tau$.

We claim that $\jo{\infty}(\hBshh{-\afunc})=\jo{\infty}(\id)$.
Evidently, it suffices to show that 
$\jo{\dg+l}(\hBshh{-\afunc})=\jo{\dg+l}(\id)$
for arbitrary large $l$.

For $l\geq0$ put $\afunc_{l} = \sum\limits_{i=0}^{l}\omega_i$ and $\afunc_{>l} = \afunc - \afunc_{l}$. 
Then
$$
\begin{array}{rcl}
\hBshh{-\afunc}(x) & = & \BFlow(\difM(x),-\afunc_{l}(x)-\afunc_{> l}(x)) = \\ [1.5mm]
& = & \BFlow\bigl(\BFlow(\difM(x),-\afunc_{l}(x)), -\afunc_{> l}(x)\bigr) = \\ [1.5mm]
& = & \BFlow(\difM_{l+1}(x), -\afunc_{> l}(x)) = 
\qBsh{\difM_{l+1}}{-\afunc_{> l}}.
\end{array}
$$
Notice that by~\eqref{equ:jets_of_hl} $\jo{\dg+l}(\difM_{l+1})=\jo{\dg+l}(\id)$.
Moreover, since $\jo{l}(\afunc_{>l})=0$, it follows from Lemma~\ref{lm:init-terms-shifts} that 
$\jo{\dg+l}(\Bsh{\afunc_{> l}})=\jo{\dg+l}(\id)$ for $l\geq1$.
Thus 
$\jo{\dg+l}(\difM_{l+1})= \jo{\dg+l}(\Bsh{\afunc_{> l}}) = \jo{\dg+l}(\id).$
Then by (2) of Corollary~\ref{cor:rel_betw_shifts}
$\jo{\dg+l} (\hBshh{-\afunc}) =\jo{\dg+l}(\qBsh{\difM_{l+1}}{-\afunc_{> l}}) =\jo{\dg+l} (\id)$.
\qed

\section{Borel's theorem}\label{sect:Borel_th}
In this section we present a variant of a well-known theorem of E.~Borel.
It will be used in the next section for the construction of $\jo{\infty}$-sections of a shift map.

Let $\Vman$ be an open subset of $\RRR^n$ and $$f=(f_1,\ldots,f_m):\Vman\to\RRR^m$$ be a smooth mapping.
For every compact $K\subset \Vman$ and $r\geq 0$ define the \emph{$r$-norm} of $f$ on $K$ by
$\|f\|^{r}_{K} =\sum\limits_{j=1}^{m} \sum\limits_{|i|\leq r} \sup\limits_{x\in K} |D^{i}f_j(x)|,$
where $i=(i_1,\ldots,i_n)$, $|i|=i_1+\cdots+i_n$, and
$D^i = \frac{\partial^{|i|}}{\partial x_1^{i_1}\cdots\partial x_{n}^{i_n}}.$
For a fixed $r$ the norms $\|f\|^{r}_{K}$, where $K$ runs over all compact subsets of $\Vman$, define the \emph{weak Whitney $C^{r}_{W}$ topology\/} on $C^{\infty}(\Vman,\RRR^m)$.
A \emph{$C^{\infty}_{W}$-topology\/} on $C^{\infty}(\Vman,\RRR^m)$ is generated by $C^{r}_{W}$-topologies for all finite $r\geq0$.

For every $i\geq0$ denote by $\Poly{i}$ the space of real homogeneous polynomials in $n$ variables $x_1,\ldots,x_n$ of degree $i$.
Associating to every $\omega\in\Poly{i}$ its coefficients we can identify $\Poly{i}$ with $\RRR^{C^{i-1}_{n+i-1}}$, where $C^{i-1}_{n+i-1}=\frac{(n+i-1)!}{n!\,(i-1)!}$.
On the other hand $\Poly{i} \subset C^{\infty}(\RRR^n,\RRR)$. 
It is easy to show that every of $C^{r}_{W}$-topologies on $\Poly{i}$ induced from $C^{\infty}(\RRR^n,\RRR)$ coincides with the Euclidean one.

Let $A,B$ be smooth manifolds and $\XX \subset C^{\infty}(A,B)$ be a subset.
We will say that a map $\lambda:\XX\to\Poly{i}$ is $C^{r}_{W}$-continuous provided it is continuous from $C^{r}_{W}$-topology of $\XX$ to the Euclidean topology of $\Poly{i}$.

\begin{theorem}\label{th:Borel}
Let $A,B$ be smooth manifolds, $\XX \subset C^{\infty}(A,B)$ a subset, and $\Asubm\subset \RRR^n$ an open neighbourhood of $\orig\in \RRR^n$.
Suppose that for every $i\geq0$ we are given a number $s_i\geq0$ and a $C^{s_i}_{W}$-continuous map $$\lambda_{i}:\XX \to \Poly{i}.$$ 
Thus we can define the following mapping
$$\lambda: \XX \to \RRR[[x_1,\ldots,x_n]],\qquad \lambda(\afunc) = \sum_{i=0}^{\infty}\lambda_{i}(\afunc).$$
Then there exists a $C^{\infty,\infty}_{W,W}$-continuous map
$\SCT: \XX \to C^{\infty}(\Asubm,\RRR)$ such that $\infty$-jet of $\SCT(\afunc)$ at $\orig$ coincides with $\lambda(\afunc)$ for every $\afunc\in \XX$.

Moreover if every $\lambda_i$ preserves smoothness (in the sense of Definition~\ref{defn:pres_smoothness}), then so does $\SCT$.
\end{theorem}
\begin{proof}
The proof uses a theorem of E.~Borel claiming that there exists a map $\BB:\RRR[[x_1,\ldots,x_n]]\to C^{\infty}(\Asubm,\RRR)$ such that for every formal series $\tau\in\RRR[[x_1,\ldots,x_n]]$ the Taylor series of the function $\BB(\tau)$ at $\orig$ coincides with $\tau$, see e.g.~\cite{GolubitskyGuillemin}.

Such a map $\BB$ can be assumed to have the following properties (a)-(c) below.
Let $\tau=\mathop\sum\limits_{i=0}^{\infty}\omega_i$ be a formal series, where $\omega_i\in\Poly{i}$.
Then 
\begin{enumerate}
\item[(a)] 
for every $i\geq0$ the restriction $\BB|_{\Poly{i}}: \Poly{i} \to C^{\infty}(\Asubm,\RRR)$ is continuous from the Euclidean topology of $\Poly{i}$ to $C^{\infty}_{W}$-topology of $C^{\infty}(\Asubm,\RRR)$;
\item[(b)] 
$\BB(\mathop\sum\limits_{i=0}^{\infty}\omega_i) = \mathop\sum\limits_{i=0}^{\infty}\BB(\omega_i)$;
\smallskip
 \item[(c)]
   for every compact $K\subset\Vman$, $\eps>0$, and an integer number $r\geq0$ there exists $s=s(K,\eps,r)\geq0$ which does not depend on $\tau$ such that
$$\left\|\BB\left(\mathop\sum\limits_{i=s+1}^{\infty}\omega_i\right)\right\|^{r}_{K} < \eps.$$
\end{enumerate}

Indeed, for every $i\geq0$ let $\rho_i:\RRR^n\to\RRR$ be a smooth function supported in a sufficiently small neighbourhood $\Vman_i\subset V$ of $\orig$ and equal to $1$ in a smaller neighbourhood of $0$.
Put $\BB(\omega)=\omega \rho_i$ for all $\omega\in\Poly{i}$. 
Then (a) holds true.

Property (b) is just the definition of $\BB$ on all of $\RRR[[x_1,\ldots,x_n]]$.

Finally, (c) can be reached if the supports of $\rho_i$ decrease sufficiently fast when $i$ tends to $\infty$, see for details e.g.~\cite{GolubitskyGuillemin}.

\medskip 

Assuming that $\BB$ has properties (a)-(c) we will now show that the following map
\begin{equation}\label{equ:Lambda}
\begin{array}{c}
\SCT = \BB \circ \lambda: \XX \xrightarrow{~\lambda~} \RRR[[x_1,\ldots,x_n]]  \xrightarrow{~\BB~}C^{\infty}(\Asubm,\RRR), \\ [2mm]
\SCT(q)(x) = \sum\limits_{i=0}^{\infty} \rho_i(x) \lambda_i(q)(x), \qquad q\in \XX, \ x\in\Vman,
\end{array}
\end{equation}
satisfies the statement of our theorem.

Indeed, for each $\afunc\in\XX$ the Taylor series of $\SCT(\afunc) =\BB\circ\lambda(\afunc)$ at $\orig$ coincides with $\lambda(\afunc)$.

Let us verify $\contWW{\infty}{\infty}$-continuity of $\SCT$ at $\afunc$.
It suffices to prove that for every compact subset $K\subset\Vman$, $\eps>0$, and $r\geq0$ there exist $c=c(K,\eps,r)\geq0$ which does not depend on $\afunc$, and a $C^{c}_{W}$-neighbourhood $\Nbh_{\afunc}$ of $\afunc$ in $\XX$ such that $\|\SCT(\afunc)-\SCT(\bfunc)\|^{r}_{K}<\eps$ for all $\bfunc\in\Nbh_{\afunc}$.

For every $s\geq0$ define the following two maps:
$$\SCT_{s}, \SCT_{>s}:\XX \to C^{\infty}(\Asubm,\RRR),$$
$$
\SCT_{s}(\afunc) = \BB \circ \mathop\oplus\limits_{i=0}^{s} \lambda_i(\afunc), \qquad
\SCT_{>s}= \SCT - \SCT_{s}.
$$
Then 
\begin{equation}\label{equ:Lambda_est}
\|\SCT(\afunc)-\SCT(\bfunc)\|^{r}_{K} \leq
\|\SCT_{s}(\afunc)-\SCT_{s}(\bfunc)\|^{r}_{K}+
\|\SCT_{>s}(\afunc)\|^{r}_{K} +
\|\SCT_{>s}(\bfunc)\|^{r}_{K}.
\end{equation}

For every $s\geq0$ let $c(s)=\max\{s_i\}_{i=0}^{s}$.
Then it follows from (a) and assumptions about continuity of $\lambda_i$ that $\SCT_{s}$ is $\contWW{c(s)}{\infty}$-continuous.
Hence $\SCT_{s}$ is $\contWW{c(s)}{\rsm}$-continuous for every $\rsm\geq0$.

Fix $\eps>0$, $r\geq0$, and a compact subset $K\subset \Vman$.
Then by (c) there exists $s>0$ which does not depend on $\afunc$ and $\bfunc$ such that each of the last two terms in~\eqref{equ:Lambda_est} is less that $\frac{\eps}{3}$.
Moreover it follows from $\contWW{c(s)}{\rsm}$-continuity of $\SCT_{s}$ that there exists a $C^{c(s)}_{W}$-neighbourhood $\Nbh_{\afunc}$ of $\afunc$ in $\XX$ such that for every $\bfunc\in\Nbh_{\afunc}$ the first term on the right side of~\eqref{equ:Lambda_est} is less than $\frac{\eps}{3}$ as well.
Then $\|\SCT(\afunc)-\SCT(\bfunc)\|^{r}_{K} < \eps$.

Suppose now that every $\lambda_i$ preserves smoothness.
Let $q:A\times\RRR^k\to B$ be a $C^{\infty}$ map such that $q_t\in \XX$ for every $t\in\RRR^k$.
Then for every $i\geq0$ the following map
$$
\lambda_i(q):\Vman\times\RRR^{k}\to\RRR,
\qquad
\lambda_i(q)(x,t)=\lambda_i(q_t)(x)
$$
is also $C^{\infty}$.
Therefore it follows from~\eqref{equ:Lambda} that the following map
$\SCT(q):\Vman\times\RRR^{k}\to\RRR$ defined by 
$$
\SCT(q)(x,t)=\SCT(q_t)(x)=
\sum_{i=0}^{\infty} \rho_i(x) \lambda_i(q_t)(x)
$$
is $C^{\infty}$ as well.
Hence $\SCT$ preserves smoothness.
\end{proof}

\section{$\jo{\infty}$-sections of the shift map}\label{sect:cont-corresp}
\subsection{Notation.}
Let $\Vman$ be an open neighbourhood of $\orig$ in $\RRR^n$, $\BFld$ be a vector field on $\Vman$ such that $\BFld(\orig)=0$, $\BFlow:\Vman\times\RRR\supset\BDom \to\RRR^n$ be the local flow of $\BFld$, $\funcBV$ be the subset of $C^{\infty}(\Vman,\RRR)$ consisting of functions $\afunc$ whose graph is contained in $\BDom$, and 
$$
\Shift: \funcBV \to C^{\infty}(\Vman,\RRR^n), \qquad
\Shift(\afunc)(z) = \Bsh{\afunc}(x) = \BFlow(z,\afunc(z))
$$
be the corresponding shift map of $\BFld$.
Denote its image by $\imShBV$.

For every $k\geq 0$ and $\difM\in C^{\infty}(\Vman,\RRR^n)$ let $\jo{k}(\difM)$ be the $k$-jet of $\difM$ at $\orig$.

Let also $\JShBV \subset C^{\infty}(\Vman,\RRR^{n})$ be the subset consisting of maps $\difM$ for which there exists a smooth function $\afunc_{\difM}\in C^{\infty}(\Vman,\RRR)$ such that $\jo{\infty}(\difM)=\jo{\infty}(\Bsh{\afunc_{\difM}})$, c.f.~\eqref{equ:gJShBO}.

In this section we will show how to choose $\afunc_{\difM}$ so that the correspondence $\difM\mapsto\afunc_{\difM}$ becomes a continuous and preserving smoothness map.
Such a map will be called a $\jo{\infty}$-section of $\ShBV$.

\begin{definition}\label{def:j-section}
Let $\XX\subset \JShBV$ be a subset.
We will say that a mapping $\JS:\XX \to C^{\infty}(\Vman,\RRR)$ is a {\bfseries $\jo{\infty}$-section\/} of $\ShBV$ on $\XX$ provided $\JS$ is $C^{\infty,\infty}_{W,W}$-continuous, preserves smoothness, and
\begin{equation}\label{equ:jih_jiFLh}
\jo{\infty}(\difM)=\jo{\infty}(\ShBV\circ \JS(\difM)), \qquad \forall \difM\in\XX.
\end{equation}
If $\XX=\JShBV$ then $\JS$ will be called a {\bfseries global} $\jo{\infty}$-section of $\ShBV$.
\end{definition}

Recall that we use the following notation
$$\ShBV\circ \JS(\difM)(x) = \Bsh{\JS(\difM)}(x) = \BFlow(x,\JS(\difM)(x)),$$
$$\hBshh{-\JS(\difM)}(x)=\BFlow(\difM(x),-\JS(\difM)(x)).$$
Then by (2) of Corollary~\ref{cor:rel_betw_shifts} relation~\eqref{equ:jih_jiFLh} is equivalent to each of the following conditions
\begin{equation}\label{equ:equiv-jinfty-sect}
\jo{\infty}(\difM)=\jo{\infty}(\Bsh{\JS(\difM)})
\qquad\Leftrightarrow\qquad
\jo{\infty}(\hBshh{-\JS(\difM)})=\jo{\infty}(\id).
\end{equation}

Our main result proves existence of $\jo{\infty}$-sections on certain subspaces of $\JShBV$, see Theorem~\ref{th:cont-estimations}.
Before formulating this theorem, let us show how $\jo{\infty}$-sections can be used for constructing real sections of $\ShBV$.

\subsection{Applications of $\jo{\infty}$-sections.}\label{sect:appl_jinf_sect}
Denote by $\EBVi \subset \EBV$ the subset consisting of maps $\difM$ such that $\jo{\infty}(\difM)=\jo{\infty}(\id)$.
Since $\jo{\infty}(\difM)=\jo{\infty}(\Bsh{0})$, we obtain that $$\EBVi \;\subset\; \EBV\cap\JShBV.$$

Let $\XX\subset\EBV\cap\JShBV$ be a subset.
Thus every $\difM\in\XX$ is an orbit preserving map $\Vman\to\RRR^{n}$ being a diffeomorphism at every singular point $z\in\singB$ and such that $\jo{\infty}(\difM)=\jo{\infty}(\Bsh{\afunc_{\difM}})$ for some $\afunc_{\difM}\in C^{\infty}(\Vman,\RRR)$.

Suppose that there exists a $\jo{\infty}$-section $\JS:\XX\to C^{\infty}(\Vman,\RRR)$ of $\ShBV$.
Then by~\eqref{equ:equiv-jinfty-sect} $\hBshh{-\JS(\difM)}\in\EBVi$ for every $\difM\in\XX$, whence the following map is well defined:
\begin{equation}\label{equ:EJ_map}
\EJ:\XX \to \EBVi,
\qquad
\EJ(\difM)=\hBshh{-\JS(\difM)}.
\end{equation}
\begin{lemma}\label{lm:inf_sect_alm_sect}{\rm c.f.~\cite[Pr.~3.4]{Maks:Hamvf:2006}}
Suppose also that there exists a $C^{\infty,\infty}_{W,W}$-continuous and preserving smoothness {\bfseries section\/} $$\SI:\EBVi\to C^{\infty}(\Vman,\RRR)$$ of $\ShBV$, i.e. $\difM(x)=\BFlow(x,\SI(\difM)(x))$ for all $\difM\in\EBVi$.
Then the following map $\RS:\XX \to C^{\infty}(\Vman,\RRR)$ defined by
$$
\RS(\difM)=
\JS(\difM) + \SI\circ\EJ(\difM) = 
\JS(\difM) + \SI(\hBshh{-\JS(\difM)}).
$$
is a section of $\ShBV$ defined on all of $\XX$.
\end{lemma}
\proof
This statement was actually established in~\cite[Pr.~3.4]{Maks:Hamvf:2006}.
For the convenience of the reader we recall the proof.
It suffices to show that $\BFlow(\difM(x),-\RS(\difM)(x))\equiv x$ for all $x\in\Vman$ and $\difM\in\XX$ but this easily follows from definitions:
$$
\begin{array}{rcl}
\BFlow(\difM(x),-\RS(\difM)(x)) & = & \BFlow\bigl(\difM(x),-\JS(\difM)(x) - \SI\circ\EJ(\difM)(x)\bigr)\\ [2mm]
& = & \BFlow\bigl(\BFlow(\difM(x),-\JS(\difM)(x)), -\SI\circ\EJ(\difM)(x)\bigr) \\ [2mm]
& = & \BFlow\bigl(\EJ(\difM)(x), -\SI\circ\EJ(\difM)(x)\bigr) = x. \ \qquad \qed
\end{array}
$$

Thus existence of $\jo{\infty}$-sections reduces the problem of resolving~\eqref{equ:h_is_a_shift} to the case when $\jo{\infty}(\difM)=\jo{\infty}(\id)$.

\subsection{Main result.}
Suppose that $\BFld$ is not flat at $\orig$, i.e.\! there exists $\dg\geq1$ such that $\jo{\dg-1}(\BFld)=0$ and $\PFunc=\jo{\dg}(\BFld)$ is a non-zero homogeneous vector field.
For $\dg=1$ we will write $\PFunc(x) = \Lmatr\cdot x$, where $\Lmatr$ is a certain non-zero $(n\times n)$-matrix.
In this case we have the exponential map 
$$\exp_{\Lmatr}:\RRR\to \GL(\RRR,n), \qquad \exp_{\Lmatr}(t) = e^{\Lmatr\, t}.$$ 
Denote its image by $E_{\Lmatr}=\{e^{\Lmatr\, t}\}_{t\in\RRR}$.
Then the following three cases of $E_{\Lmatr}$ will be separated:
\begin{enumerate}
\item[\coiA]
$E_{\Lmatr}\approx \RRR$ and is a \emph{closed\/} subgroup of $\GL(\RRR,n)$;
\item[\coiB]
$E_{\Lmatr}\approx SO(2)$;
\item[\coiC]
$E_{\Lmatr}\approx \RRR$ and is a \emph{non-closed\/} subset of $\GL(\RRR,n)$.
\end{enumerate}

For every $k\geq 1$ denote
$$
\JShBVi{k} = \bigl\{ 
\difM\in\JShBV \ | \  \jo{k-1}(\difM) = \jo{k-1}(\id) \bigr\}, \qquad k\geq1.
$$
Since $\difM(\orig)=\orig$, i.e. $\jo{0}(\difM)=\jo{0}(\id)$, for all $\difM\in\JShBV$, we see that $\JShBV=\JShBVi{1}$.
Moreover, it follows from Lemma~\ref{lm:init-terms-shifts} that $\jo{\dg-1}(\difM)=\jo{\dg-1}(\id)$ for $\dg\geq2$.
Hence for arbitrary $\dg\geq1$ we have the following inclusions:
$$
\JShBV=\JShBVi{1} = \cdots=  \JShBVi{\dg} \supset \JShBVi{\dg+1} \supset \cdots
$$

It will also be convenient to define local variants of the above constructions.
Recall that $\gJShBO$ is the subgroup of $\DiffRno$ consisting of space of germs at $\orig$ of maps from $\gJShBO$, see~\eqref{equ:gJShBO}.
Therefore we define $\gJShBOi{k}$ to be the subspace of $\gJShBO$ consisting of germs at $\orig$ of maps from $\JShBVi{k}$.
Then again
$$
\gJShBO=\gJShBOi{1} = \cdots=  \gJShBOi{\dg} \supset \gJShBOi{\dg+1} \supset \cdots
$$

\begin{theorem}\label{th:cont-estimations}
{\rm(1)} There exists a $\jo{\infty}$-section $\JS$ of $\ShBV$ on $\JShBVi{2}$.
Hence if $\dg\geq2$, then $\JShBV=\JShBVi{2}$, and thus $\JS$ is defined on all of $\JShBV$.

{\rm(2)} 
If $\dg=1$ and $E_{\Lmatr}$ satisfies \coiA, then there exists a global $\jo{\infty}$-section $\JS$ of $\ShBV$ as well.

{\rm(3)}
If $\dg=1$ and $E_{\Lmatr}$ satisfies \coiB, then for every $g\in\JShBV$ the shift map $\ShBV$ has a $\jo{\infty}$-section defined on some $C^{1}_{W}$-neighbourhood of $g$ in $\JShBV$.
\end{theorem}

\begin{remark}\rm
If $\dg=1$ but $E_{\Lmatr}$ satisfies \coiC\ then it seems that $\ShBV$ has no even local ($C^{\infty,\infty}_{W,W}$-continuous) $\jo{\infty}$-sections on $\JShBVi{1}$, though by (1) it has a $\jo{\infty}$-section on $\JShBVi{2}$.
\end{remark}

The proof will be given in~\S\ref{sect:proof:th:cont-estimations}.
It is similar to the proof Theorem~\ref{th:grp_subset_JSh} but we have to estimate continuity of certain correspondences.
We need the following two statements.

\begin{proposition}\label{pr:jinfsect_1}
Let $\dg=1$.
In the case \coiA\ of $E_{\Lmatr}$ there exists a $C^{1}_{W}$-continuous map $\Delta_{0}:\JShBV\to \RRR$ such that for every $\difM\in\JShBV$ the germ at $\orig$ of the mapping 
$$ \qBsh{\difM}{-\Delta_{0}(\difM)}(x) = \BFlow(\difM(x),-\Delta_{0}(\difM))$$ belongs to $\gJShBOi{2}$.

Suppose that $E_{\Lmatr}$ satisfies \coiB.
Then for every $g\in\JShBV$ there exists a $C^{1}_{W}$-neighbourhood $\NN_{g}$ in $\JShBV$ and a $C^{1}_{W}$-continuous mapping $\Delta_{0}:\NN_{g}\to \RRR$ such that for every $\difM\in\NN_{g}$ the germ at $\orig$ of the mapping $\qBsh{\difM}{-\Delta_{0}(\difM)}$ belongs to $\gJShBOi{2}$. 

In both cases $\Delta_{0}$ preserves smoothness.
\end{proposition}
\begin{proof}
Let $\jo{1}:\JShBV\to\GL(\RRR,n)$ be the map associating to every $\difM\in\JShBV$ its Jacobi matrix $J(\difM,\orig)$ at $\orig$.
Then it follows from Lemma~\ref{lm:init-terms-shifts} that the image of $\jo{1}$ coincides with $E_{\Lmatr}$.
Thus we have two maps:
$$
\begin{CD}  \JShBV  @>{\jo{1}}>> E_{\Lmatr} @<{\exp_{\Lmatr}}<< \RRR. \end{CD}
$$
Notice that in the cases \coiA\ and \coiB\ $E_{\Lmatr}$ is a closed subgroup of $\GL(\RRR,n)$.
Moreover, in the case \coiA\ $\exp_{\Lmatr}$ is an embedding, so we can put
$$\Delta_0:\JShBV\to\RRR, \qquad\Delta_0=\exp_{\Lmatr}^{-1} \circ \jo{1}.$$
In the case \coiB\ $\exp_{\Lmatr}$ is a smooth $\ZZZ$-covering map, whence for every $g\in\JShBV$ there exists only a $C^{1}_{W}$-neighbourhood $\NN_{g}$ in $\JShBV$ such that the map $\Delta_0=\exp_{\Lmatr}^{-1} \circ \jo{1}:\NN_{g}\to\RRR$ is well-defined.

It is easy to see that in both cases $\Delta_0$ has the desired properties.
\end{proof}

\begin{proposition}\label{pr:jinfsect_2}
Let $\dg+l\geq 2$.
Then there exists a $C^{\dg+l}_{W}$-continuous and preserving smoothness map $\Delta_{l}:\JShBVi{\dg+l} \to \Poly{l}$ such that for every $\difM\in\JShBVi{\dg+l}$ the germ at $\orig$ of the map $$\qBsh{\difM}{-\Delta_{l}(\difM)}(x)=\BFlow(\difM(x),-\Delta_{l}(\difM)(x))$$  belongs to $\gJShBOi{\dg+l+1}$.
\end{proposition}
\begin{proof}
Let $\difM\in\JShBVi{\dg+l}$.
Since $\dg+l\geq2$, we have that either $\dg\geq2$ or $\dg=1$ but $l\geq1$.
It follows from Lemma~\ref{lm:init-terms-shifts} that in both cases there exists a \emph{unique} homogeneous polynomial $\omega_l$ of degree $l$ such that $\jo{\dg+l}(\difM)(x) =x + \Qmap(x) \cdot \omega_{l}(x)$.
The correspondence $\difM\mapsto\omega_l$ is a well-defined map $\Delta_l:\JShBVi{\dg+l} \to \Poly{l}$ and by (2) of Corollary~\ref{cor:rel_betw_shifts} the germ at $\orig$ of the mapping $\qBsh{\difM}{-\Delta_{l}(\difM)}$ belongs to $\gJShBOi{\dg+l+1}$.

Let us verify the continuity of $\Delta_{l}$.
Consider the following map
$$
\jo{\dg+l}:C^{\infty}(\Vman,\RRR^n)\to J^{\dg+l}(\Vman,\RRR^n),
\qquad \difM \mapsto \jo{\dg+l}(\difM),
$$
associating to each $\difM\in C^{\infty}(\Vman,\RRR^n)$ its $(\dg+l)$-jet $\jo{\dg+l}(\difM)$ at $\orig$.
Evidently, $\jo{\dg+l}$ is a $C^{\dg+l}_{W}$-continuous and preserving smoothness map
Moreover, the image $\jo{\dg+l}(\EndBVi{\dg+l})$ is contained in the following set
$\mathcal{A}_{l} = \{ \, x + \Qmap(x)\cdot\omega(x)\  | \ \omega\in\Poly{l} \, \} \; \subset \; J^{\dg+l}(\Vman,\RRR^n).$
Further, it follows from smoothness of the Euclid algorithm of division of polynomials that the correspondence
\ $x + \Qmap(x)\cdot\omega(x) \mapsto \omega$ \
is a well-defined smooth map $D:\mathcal{A}_{l}\to\Poly{l}$.
Therefore $\Delta_l = D \circ \jo{\dg+l}$ is $C^{\dg+l}_{W}$-continuous and preserves smoothness.
\end{proof}

\subsection{Proof of Theorem~\ref{th:cont-estimations}.}\label{sect:proof:th:cont-estimations}
Let $\XX$ be one of the spaces $\JShBVi{2}$, $\JShBV$, or $\NN_{g}$ with respect to the cases (1), (2), or (3) of our theorem, where $\NN_{g}$ is a $C^{1}_{W}$-neighbourhood of $g\in\JShBV$ constructed in Propositions~\ref{pr:jinfsect_1}.

Then similarly to the proof of Theorem~\ref{th:grp_subset_JSh} for every $\difM\in\XX$ we can construct a sequence of homogeneous polynomials $\{\omega_i\}_{i=0}^{\infty}$, $(\deg\omega_i=i)$ via the following rule:
$$
\omega_0 = \Delta_0(\difM),
\quad
\omega_1=\Delta_1(\hBshh{-\omega_0}), 
\quad
\ldots,
\quad
\omega_l = \Delta_l\bigl(\hBshh{-\sum_{i=0}^{l-1}\omega_i}\bigr).
$$
It follows from Propositions~\ref{pr:jinfsect_1}, \ref{pr:jinfsect_2}, and formulae for $\omega_i$ that for every $i\geq0$ the correspondence $\difM\mapsto\omega_i$ is a $C^{\dg+l}_{W}$-continuous and preserving smoothness map
$$\lambda_i:\XX\to\Poly{i}, \qquad \lambda_i(\difM)=\omega_i.$$
Put $\lambda(\difM)=\sum_{i=0}^{\infty}\lambda_i(\difM)$.
Then $\jo{\infty}(\hBshh{-\lambda(\difM)})=\jo{\infty}(\id)$ as well as in Theorem~\ref{th:grp_subset_JSh}.

Now it follows from Borel's Theorem~\ref{th:Borel} applied to $\XX$ that there exists a $C^{\infty,\infty}_{W,W}$-continuous and preserving smoothness map 
$$\SCT: \XX \to C^{\infty}(\Vman,\RRR)$$
such that $\jo{\infty}(\SCT(\difM))=\lambda(\difM)$ for $\difM\in\XX$.
Hence $\jo{\infty}(\hBshh{-\SCT(\difM)})=\jo{\infty}(\id)$.
Then by (2) of Corollary~\ref{cor:rel_betw_shifts} 
$\jo{\infty}(\difM)=\jo{\infty}(\Bsh{\SCT(\difM)})$ for all $\difM\in\XX$.
\qed

\section{Property \AST}\label{sect:prop-AST}
In this section we describe a class of vector fields $\BFld$ on $\RRR^n$ for which $\jo{\infty}\gimShB=\jo{\infty}\gEndIdBzr{1}$, see Theorem~\ref{th:AST_implies_infjets}.
This class is rather special since it consists of \emph{completely integrable} (i.e.\! having $n-1$ almost everywhere independent integrals) vector fields satisfying certain non-degeneracy conditions.

\subsection{Cross product.}
Let $(x_1,\ldots,x_n)$ be coordinates in $\RRR^n$.
For every smooth function $\BFunc:\RRR^n\to\RRR$ define the following ``gradient'' vector field with respect to these coordinates:
$$
\grad_x\BFunc = (\BFunc'_{x_1}, \ldots, \BFunc'_{x_n}).
$$

If $\BFunc_1,\ldots,\BFunc_{n-1}:\RRR^n\to\RRR$ is an $(n-1)$-tuple of smooth functions, then we can define the following \emph{cross-product} vector field:
\begin{equation}\label{equ:cross-prod-formula}
\HFld=[\nabla_x \BFunc_1, \ldots,\nabla_x \BFunc_{n-1}] = 
\left|
\begin{array}{cccc}
\frac{\partial \BFunc_1}{\partial x_1} & \frac{\partial \BFunc_1}{\partial x_2} & \cdots & \frac{\partial \BFunc_1}{\partial x_n} \\ [2mm]
\cdots       & \cdots  & \cdots  & \cdots \\
\frac{\partial \BFunc_{n-1}}{\partial x_1} & \frac{\partial \BFunc_{n-1}}{\partial x_2} & \cdots & \frac{\partial \BFunc_{n-1}}{\partial x_n} \\  [4mm]
\frac{\partial }{\partial x_1} & \frac{\partial }{\partial x_2} & \cdots & \frac{\partial }{\partial x_n} 
\end{array}
\right|
\end{equation}
being an analogue of the cross-product $[a,b]$ of two vectors $a,b$ in $\RRR^3$.
Notice that the first $n-1$ rows of this $(n\times n)$-matrix consist of smooth functions, while the $n$-th row is the standard basis 
$\langle \frac{\partial}{\partial x_1},\ldots,\frac{\partial}{\partial x_{n}}\rangle$ of the space of vector fields on $\RRR^n$. 
Therefore the corresponding determinant is a well-defined vector field.

Equivalently, let us fix the standard Euclidean metric on $\RRR^n$.
Then we have a Hodge isomorphism $*:\JS^{n-1}(\RRR^n) \to \JS^{1}(\RRR^n)$ between the spaces of differential forms and the isomorphism $\phi:\JS^1(\RRR^n)\to \Gamma(\RRR^n)$ between the space of $1$-forms and the space of vector fields on $\RRR^n$.
Then it is easy to see that
$$
[\nabla_x \BFunc_1, \ldots,\nabla_x \BFunc_{n-1}] = 
\phi \circ * (d\BFunc_1 \wedge \cdots \wedge d\BFunc_{n-1}).
$$

It is easy to see that every $\BFunc_i$ is constant along orbits of $\HFld$, i.e.\! $\BFunc_i$ is an \emph{integral\/} for $\HFld$.
Indeed, substituting $\grad_x\BFunc_i$ in~\eqref{equ:cross-prod-formula} instead of last row, we will get $d\BFunc_i(\HFld)=0$.
Thus $\HFld$ is \emph{completely integrable} in the sense that it has $n-1$ integrals and its singular set coincides with the set of points where the gradients $\grad_x\BFunc_1,\ldots,\grad_x\BFunc_{n-1}$ are linearly dependent.

\begin{example}\rm
Let $\BFunc:\RRR^2 \to\RRR$ be a smooth function.
Then
$$\HFld=[\grad_x \BFunc] =\left| 
\begin{matrix}
\BFunc'_{x}  & \BFunc'_{y} \\
\frac{\partial }{\partial x} & \frac{\partial }{\partial y} 
\end{matrix}
\right| = 
- \BFunc'_{y} \, \frac{\partial }{\partial x}  \; + \; 
\BFunc'_{x} \, \frac{\partial }{\partial y}
$$
is the corresponding Hamiltonian vector field of $\BFunc$.
\end{example}
\begin{lemma}\label{lm:Hy_Hx}
Let $x=(x_1,\ldots,x_n)$ and $y=(y_1,\ldots,y_n)$ be two local coordinate systems at $\orig$ related by a germ of diffeomorphism $\difM=(\difM_1,\ldots,\difM_n)$ of $(\RRR^n,\orig)$, i.e.\! $x=\difM(y)$.
Let also $\HFld_{x}$ and $\HFld_{y}$ be vector fields defined by~\eqref{equ:cross-prod-formula} in the coordinates $(x_i)$ and $(y_i)$ respectively, and $\difM^{*}\HFld_{x}$ be the vector field induced by $\difM$, i.e.\! this is $\HFld_{x}$ in the coordinates $(y_i)$.
Then
$$
\HFld_{y} = |J(\difM)| \cdot \difM^{*}\HFld_{x}.
$$
\end{lemma}
\begin{proof}
Notice that if $\HFld_x(x) = \sum\limits_{i=1}^{n} T_i(x) \frac{\partial}{\partial x_i}$, then in the coordinates $(y_i)$ we can also write
$$\HFld_x(y) = \sum\limits_{i=1}^{n} T_i(y) \frac{\partial}{\partial x_i},$$
where $\frac{\partial}{\partial x_i}=\sum\limits_{j=1}^{n} \frac{\partial y_i}{\partial x_j} \frac{\partial}{\partial y_i}.$
Hence
\begin{equation}\label{equ:H_x_y}
\HFld_x(y)  = 
\left|
\begin{array}{ccc}
\frac{\partial \BFunc_1}{\partial x_1}(y)  & \cdots & \frac{\partial \BFunc_1}{\partial x_n}(y) \\ [2mm]
\cdots         & \cdots  & \cdots \\
\frac{\partial \BFunc_{n-1}}{\partial x_1}(y) & \cdots & \frac{\partial \BFunc_{n-1}}{\partial x_n}(y) \\  [4mm]
\frac{\partial }{\partial x_1} & \cdots & \frac{\partial }{\partial x_n} 
\end{array}
\right|
\end{equation}
On the other hand,
\begin{align*}
 \HFld_y(y) & = 
\left|
\begin{array}{ccc}
\frac{\partial \BFunc_1}{\partial y_1}(y)  & \cdots & \frac{\partial \BFunc_1}{\partial y_n}(y) \\ [2mm]
\cdots         & \cdots  & \cdots \\
\frac{\partial \BFunc_{n-1}}{\partial y_1}(y) & \cdots & \frac{\partial \BFunc_{n-1}}{\partial y_n}(y) \\  [4mm]
\frac{\partial }{\partial y_1} & \cdots & \frac{\partial }{\partial y_n} 
\end{array}
\right|    \\ 
&  = 
\left|
\begin{array}{ccc}
\sum\limits_{i=1}^{n}\frac{\partial \BFunc_1}{\partial x_i}(y) \cdot \frac{\partial x_i}{\partial y_1} & \cdots & 
\sum\limits_{i=1}^{n}\frac{\partial \BFunc_1}{\partial x_i}(y) \cdot \frac{\partial x_i}{\partial y_n}  \\ [2mm]
\cdots         & \cdots  & \cdots \\
\sum\limits_{i=1}^{n}\frac{\partial \BFunc_{n-1}}{\partial x_i}(y)  \cdot \frac{\partial x_i}{\partial y_1} & \cdots & 
\sum\limits_{i=1}^{n}\frac{\partial \BFunc_{n-1}}{\partial x_i}(y)  \cdot \frac{\partial x_i}{\partial y_n}  \\ [2mm]
\sum\limits_{i=1}^{n}\frac{\partial }{\partial x_i}  \cdot \frac{\partial x_i}{\partial y_1} & \cdots & 
\sum\limits_{i=1}^{n}\frac{\partial }{\partial x_i}  \cdot \frac{\partial x_i}{\partial y_n}  \\ [2mm]
\end{array}
\right|   \\
& = 
\left|
\begin{array}{ccc}
\frac{\partial \BFunc_1}{\partial x_1}(y)  & \cdots & \frac{\partial \BFunc_1}{\partial x_n}(y) \\ [2mm]
\cdots         & \cdots  & \cdots \\
\frac{\partial \BFunc_{n-1}}{\partial x_1}(y) & \cdots & \frac{\partial \BFunc_{n-1}}{\partial x_n}(y) \\  [4mm]
\frac{\partial }{\partial x_1} & \cdots & \frac{\partial }{\partial x_n} 
\end{array}
\right| 
\cdot
\left|
\begin{array}{ccc}
\frac{\partial x_1}{\partial y_1}  & \cdots & \frac{\partial x_1}{\partial y_n} \\ [2mm]
\cdots         & \cdots  & \cdots \\
\frac{\partial x_{n}}{\partial y_1} & \cdots & \frac{\partial x_n}{\partial y_n}
\end{array}
\right| \\
& \stackrel{\eqref{equ:H_x_y}}{=\!=\!=}   \HFld_x(y) \cdot |J(\difM)|.
\end{align*}
Lemma is proved.
\end{proof}

\begin{definition}\label{defn:AST-for-vf}
Let $\BFld$ be a vector field defined on some neighbourhood $\Vman$ of $\orig\in\RRR^n$.
Say that $\BFld$ {\bfseries has property \AST\ at $\orig$} if there exist $p\in\NNN$ and $n$ smooth non-flat at $\orig$ functions $$\eta,\BFunc_1,\ldots,\BFunc_{n-1}:\Vman\to\RRR$$ such that
\begin{enumerate}
 \item[\rm(a)] 
\hfil $\jo{p-1}(\BFld)=0$, \hfil
 \item[\rm(b)]
$\PFunc=\jo{p}(\BFld)$ is a non-zero homogeneous vector field being {\bfseries non-divisible by homogeneous polynomials}, i.e.\! $\PFunc$ can not be represented as a product $\PFunc=\omega\Pmap$, where $\omega$ is a homogeneous polynomial of degree $\deg\omega\geq 1$ and $\Pmap$ is a homogeneous vector field of degree $\deg\Pmap\geq1$,
 \item[\rm(c)]
vector fields $\grad_x\BFunc_1,\ldots,\grad_x\BFunc_{n-1}$ are linearly independent on an everywhere dense subset of $\Vman$ and 
\begin{equation}\label{equ:etaF_crossprod}
\eta \cdot \BFld = [\grad_x\BFunc_1,\ldots,\grad_x\BFunc_{n-1}].
\end{equation}
\end{enumerate}
\end{definition}

Let us explain this definition.

1) Since $1\leq p < \infty$, we have that $\BFld(\orig)=0$ and $\BFld$ is not flat at $\orig$.

2) We allow $\eta$ and therefore $\eta \cdot \BFld$ vanish at some points of $\Vman$ which can be even non-singular for $\BFld$.
But due to (c) the set of zeros of $\eta \cdot \BFld$ and therefore $\eta^{-1}(0)$ are nowhere dense in $\Vman$.

3) Let $k=\ord(\eta,\orig)$, $\dg_i=\ord(\BFunc_i,\orig)<\infty$, $(i=1,\ldots,n-1)$.
Thus 
\begin{equation}\label{equ:init_jets}
\gamma=\jo{k}(\eta), \qquad \GFunc_i=\jo{\dg_i}(\BFunc_i), \ (i=1,\ldots,n-1) 
\end{equation}
are non-zero homogeneous polynomials of degrees $k, \dg_1,\ldots,\dg_{n-1}$ respectively.
Then $R=[\grad_x\GFunc_1,\ldots,\grad_x\GFunc_{n-1}]$ is a homogeneous vector field of degree $\sum\limits_{i=1}^{n-1}(\dg_i-1)$ and
$\gamma\cdot \PFunc = R$.
Since $\PFunc$ is non-divisible by homogeneous polynomials, it follows that $\gamma$ is the \emph{greatest common divisor\/} of coordinate functions of $R$ in the ring $\RRR[x_1,\ldots,x_n]$.

\begin{theorem}\label{th:AST_implies_infjets}
If $\BFld$ has property \AST\ then $\jo{\infty}\gimShBO = \jo{\infty}\gEidBOr{1}.$
Moreover, for every neighbourhood $\Vman$ of $\orig$ and every $g\in \EidBVr{1}$ there exist a $C^{1}_{W}$-neighbourhood $\NN_{g}$ in $\EidBVr{1}$ and a $\jo{\infty}$-section of the shift map $\ShBV$ on $\NN_{g}$.
\end{theorem}

Thus due to Lemma~\ref{lm:inf_sect_alm_sect} in order to completely resolve~\eqref{equ:h_is_a_shift} we have to construct a section of $\ShBV$ on $\EBVi$.
This was done in~\cite{Maks:Hamvf:2006} for the case of homogeneous polynomial vector fields on $\RRR^2$ satisfying \AST, see also \S\ref{sect:vf_on_r2} for more general result.

The proof of Theorem~\ref{th:AST_implies_infjets} will be given in \S\S\ref{sect:prelim:th:AST_implies_infjets},\,\ref{sect:th:AST_implies_infjets}.
The following lemma presents a class of examples of vector fields with property \AST.
\begin{lemma}\label{lm:cross-prod-hom-poly}
Let $\BFunc_1,\ldots,\BFunc_{n-1}:\RRR^n\to\RRR$ be homogeneous polynomials such that $\grad_x\BFunc_1,\ldots,\grad_x\BFunc_{n-1}$ are linearly independent on everywhere dense subset of $\RRR^n$, and let $\eta$ be the greatest common divisor of the coordinate functions of $\HFld=[\grad_x\BFunc_1,\ldots,\grad_x\BFunc_{n-1}]$.
Then the homogeneous vector field $\BFld = \HFld/\eta$ has property \AST.
\qed
\end{lemma}

\begin{example}\rm
Let $n=2$, $\BFunc(x,y) = x^3 y^4$, and 
$$\HFld=[\grad\BFunc]=(-\BFunc'_{y}, \BFunc'_{x}) = (-4 x^3 y^3, 3 x^2 y^4) =\underbrace{x^2 y^3}_{\eta} \underbrace{(-4 x, 3 y)}_{\BFld} = \eta \BFld.$$
Then $\BFld$ is non-divisible.
Notice also that the singular set of $\HFld$ consists of $x$- and $y$-axes while the singular set of $\BFld$ is the origin only.
\end{example}

\begin{lemma}\label{lm:AST-invar}
Property \AST\ does not depend on a particular choice of local coordinates at $\orig$.
\end{lemma}
\begin{proof}
Suppose that in coordinates $x=(x_1,\ldots,x_n)$ at $\orig$ conditions (a)-(c) of Definition~\ref{defn:AST-for-vf} are satisfied.
Let $y=(y_1,\ldots,y_n)$ be another coordinates at $\orig$ related to $(x_1,\ldots,x_n)$ by a germ of a diffeomorphism $\difM=(\difM_1,\ldots,\difM_n)$ of $(\RRR^n,\orig)$, i.e.\! $x=\difM(y)$.
We have to show that conditions (a)-(c) of Definition~\ref{defn:AST-for-vf} hold in the coordinates $(y_1,\ldots,y_n)$ for the induced vector field $\difM^{*}\BFld=T\difM^{-1}\circ\BFld\circ\difM$.

Let $A=J(\difM,\orig)$ be the Jacobi matrix of $\difM$ at $\orig$.
Then it easily follows from condition (a) for $\BFld$ that 
$$
\jo{\dg-1}(\difM^{*}\BFld) = 0,
\qquad\qquad
\jo{\dg}(\difM^{*}\BFld)(y) = A^{-1} \PFunc(Ay).
$$
The latter identity implies that the initial non-zero jet of $\difM^{*}\BFld$ is non-divisible by homogeneous polynomials iff so is $\PFunc$.
This proves (a) and (b) for $\difM^{*}\BFld$.

To establish (c) apply $\difM$ to both parts of~\eqref{equ:etaF_crossprod}.
Then 
\begin{gather*}
 \difM^{*}(\eta \cdot \BFld)\;=\;\eta\circ\difM \;\cdot\; \difM^{*}\BFld, \\
 \difM^{*} [\grad_x\BFunc_1,\ldots,\grad_x\BFunc_{n-1}] \stackrel{\text{Lemma~\ref{lm:Hy_Hx}}}{=\!=\!=\!=\!=\!=\!=} 
\frac{1}{|J(\difM)|}\cdot [\grad_y\BFunc_1,\ldots,\grad_y\BFunc_{n-1}].
\end{gather*}
Denote $\eta'=\eta\circ\difM \cdot |J(\difM)|$.
Then $\eta'$ is smooth and 
$$\eta' \cdot \difM^{*}\BFld = [\grad_y\BFunc_1,\ldots,\grad_y\BFunc_{n-1}].$$
This proves (c).
\end{proof}

\section{Stabilizers of functions and polynomials}\label{sect:prelim:th:AST_implies_infjets}
In this section we present some statements which will be used in the proof of Theorem~\ref{th:AST_implies_infjets}.

\begin{lemma}\label{lm:deg_Fh-F}
Let $\BFunc\in \FuncRno$,  $\difM\in \EndRno$, and $\delta=\BFunc\circ\difM-\BFunc$. 
Suppose that $\jo{\dg-1}(\BFunc)=0$ and $\jo{k-1}(\difM)=\jo{k-1}(\id)$ for some $\dg,k\geq1$.
In particular $\Gamma=\jo{\dg}(\BFunc)$ is a homogeneous polynomial of degree $\dg$.

If $k=1$ and  $\jo{1}(\difM)(x)=Ax$ for some $(n\times n)$-matrix, then 
$$ 
\jo{\dg}(\delta)(x) = \GFunc(A\cdot x)-\GFunc(x).
$$

If $k\geq2$ and $\jo{k}(\difM)(x) =  x + \vfunc(x)$ for some homogeneous map $\vfunc$ of degree $k$, then 
$$
\jo{\dg-1+k}(\delta) = \langle \grad\GFunc, \vfunc \rangle.
$$
\end{lemma}

The proof of this lemma is direct and we leave it for the reader.
\begin{corollary}\label{cor:deg_Fh-F}
Suppose that $\difM$ preserves $\BFunc$, i.e.\! $\BFunc\circ\difM=\BFunc$, and thus $\delta\equiv0$.
If $k=1$, then $\GFunc(A\cdot x)=\GFunc(x)$.
If $k\geq2$, then $\langle \grad\GFunc, \vfunc \rangle=0$.
\end{corollary}

\subsection{Stabilizers of polynomials.}
Consider the \emph{right\/} action of the group $\GL(\RRR,n)$ on the space of polynomials $\RRR[x_1,\ldots,x_n]$ by:
\begin{equation}\label{equ:act_GL_Poly}
\begin{array}{c}
\Phi:\RRR[x_1,\ldots,x_n] \times \GL(\RRR,n) \to \RRR[x_1,\ldots,x_n] \\ [2mm]
\Phi(\GFunc, A)= \GFunc\circ A, 
\qquad\text{i.e.}\qquad 
\Phi(\GFunc, A)(x)= \GFunc(A\cdot x),
\end{array}
\end{equation}
where $(\GFunc,A)\in \RRR[x_1,\ldots,x_n] \times \GL(\RRR,n)$.
For $\GFunc\in\RRR[x_1,\ldots,x_n]$ let $$\Stab(\GFunc) = \{A\in\GL(\RRR,n) \ | \ \GFunc(A\cdot x) = \GFunc(x) \}$$ be its stabilizer with respect to $\Phi$.
Then $\Stab(\GFunc)$ is a closed (and therefore a Lie) subgroup of $\GL(\RRR,n)$.

\begin{lemma}\label{lm:tang_stab}
For every $\GFunc\in\RRR[x_1,\ldots,x_n]$ the tangent space $T_{E}\Stab(\GFunc)$ to the stabilizer $\Stab(\GFunc)$ of $\GFunc$ at $E$ consists of matrices $V\in M(\RRR,n)$  such that 
$\langle \grad\GFunc(x), V\, x \rangle=0$ for all $x\in\RRR^n$.
\end{lemma}
\begin{proof}
Let $V\in M(\RRR,n)$ and $A:\RRR\to\GL(\RRR,n)$ be the following homomorphism $A(t)=e^{Vt}$.
Evidently, $A(0)=E$ and $A'_t(0)=V$.
Notice that 
$$ 
\frac{\partial}{\partial t}\, \GFunc(A(t)\, x) = 
\langle\grad\GFunc(A(t)\, x), A'_{t}(t) \, x \rangle = 
\langle\grad\GFunc(e^{Vt}\, x), V e^{Vt} \, x \rangle.
$$ 
Then the following statements are equivalent:
\begin{enumerate}
 \item[(i)]
$V\in T_{E}\Stab(\GFunc) $;
 \item[(ii)]
$A(t)\in \Stab(\GFunc)$, i.e.\! $\GFunc(A(t)\, x) = \GFunc(x)$, for all $t\in\RRR$ and $x\in\RRR^n$;
\smallskip
 \item[(iii)]
$\frac{\partial}{\partial t}\GFunc(A(t)\, x)=\langle\grad\GFunc(e^{Vt}\, x), V e^{Vt} \, x \rangle = 0$, for all $x\in\RRR^n$;
\smallskip
 \item[(iv)]
$\frac{\partial}{\partial t}\GFunc(A(t)\, x)|_{t=0}=
\langle\grad\GFunc(x), V \, x \rangle = 0$ for all $x\in\RRR^n$.
\end{enumerate}
The implications (i)$\Leftrightarrow$(ii)$\Leftrightarrow$(iii)$\Rightarrow$(iv) are evident and (iv)$\Rightarrow$(iii) can be obtained by substituting $e^{Vt}x$ instead of $x$ in (iv).
It remains to note that our lemma claims that (i)$\Leftrightarrow$(iv).
\end{proof}

Let $\GFunc_1,\ldots,\GFunc_{n-1}\in\RRR[x_1,\ldots,x_n]$ and
\begin{equation}\label{equ:stab_G1_Gn1}
\Stab = \mathop\cap\limits_{i=1}^{n-1}\Stab(\GFunc_i)
\end{equation}
be the intersection of their stabilizers.
Then $\Stab$ is a closed Lie subgroup of $\GL(\RRR,n)$.
Denote by $T_{E}\Stab$ the tangent space of $\Stab$ at the unity matrix $E$, and let $\Stab_{E}$ be the unity component of $\Stab$.

\begin{lemma}\label{lm:Stab-G1-Gn-1}
Let $\HFld=[\grad\GFunc_1,\ldots,\grad\GFunc_{n-1}]$ be the vector field on $\RRR^n$ defined by~\eqref{defn:AST-for-vf}, $\eta$ be the greatest common divisor of coordinate functions of $\HFld$, and $\PFunc = \HFld/\eta$.
Suppose that $\PFunc\not\equiv0$.
Then $\PFunc$ is non-divisible by polynomials, i.e.\! if $\PFunc(x)=\omega(x) U(x)$, where $\omega$ is a polynomial and $U$ is a polynomial vector field, then either $\deg\omega=0$, or $\deg U=0$.

{\rm(i)} If $\deg\PFunc=1$, i.e.\! $\PFunc(x)=\Lmatr x$ for some non-zero matrix $\Lmatr\in M(\RRR,n)$, then $T_{E}\Stab=\{\Lmatr\,t\}_{t\in\RRR}$ and $\Stab_{E} = \{ e^{\Lmatr\,t} \}_{t\in \RRR}$.

{\rm(ii)} If $\deg\PFunc\geq2$ then $\Stab_{E}=\{E\}$.
\end{lemma}
\begin{proof}
Notice that $T_{E}\Stab = \mathop\cap\limits_{i=1}^{n-1} T_{E} \Stab(\GFunc_i)$.
Let $\Umatr\in T_{E}\Stab$.
Then by Lemma~\ref{lm:tang_stab} 
$\langle \grad\GFunc_i(x), \Umatr x \rangle=0$ for every $i=1,\ldots,n-1.$
Therefore $\Umatr x$ is parallel to the cross product $\HFld(x)$ of gradients $\grad\GFunc_i$ and therefore to $\PFunc(x)$ at every $x\in\RRR^n$.
If $\Umatr\not=0$, then there exists a non-zero polynomial $\omega$ such that 
\begin{equation}\label{equ:P_om_U}
\PFunc(x)=\omega(x) \cdot\Umatr\,x
\end{equation}
Since $\PFunc$ is non-divisible, this identity is possible only if $\omega$ is a constant and in this case $\deg\PFunc=1$.

Hence if $\deg\PFunc\geq 2$, then $\Umatr$ is always zero, whence $T_{E}\Stab=\{0\}$, and $\Stab_{E}=\{E\}$.
This proves (ii).

(i) Suppose that $\deg\PFunc=1$.
Then it follows from Lemma~\ref{lm:tang_stab} that $\{\Lmatr t\}_{t\in\RRR}\subset T_{E}\Stab$.
On the other hand, as noted above for every $\Umatr\in T_{E}\Stab$ there exists $\omega\in\RRR$ such that~\eqref{equ:P_om_U} holds true, whence $\Lmatr=\omega \Umatr$.
Therefore $\Umatr\in \{\Lmatr t\}_{t\in\RRR}$, and thus $\{\Lmatr t\}_{t\in\RRR}= T_{E}\Stab$.
\end{proof}

\section{Proof of Theorem~\ref{th:AST_implies_infjets}}\label{sect:th:AST_implies_infjets}
Suppose that $\BFld$ has property \AST\ at $\orig$.
Thus 
$$
\eta \cdot \BFld = [\grad_x\BFunc_1,\ldots,\grad_x\BFunc_{n-1}],
$$
where $\eta,\BFunc_1,\ldots,\BFunc_{n-1}:\Vman\to\RRR$ are germs of smooth functions satisfying assumptions of Definition~\ref{defn:AST-for-vf}.
We have to show that 
\begin{equation}\label{equ:joSh-joEidBO1}
\jo{\infty}\gimShBO = \jo{\infty}\gEidBOr{1}
\end{equation}
and for every open neighbourhood $\Vman$ of $\orig$ and $g\in\EidBVr{1}$ construct a local $\jo{\infty}$-section of $\ShBV$ defined on some $C^{1}_{W}$-neighbourhood $\NN_{g}$ of $g$ in $\EidBVr{1}$.

\subsection*{Proof of~\eqref{equ:joSh-joEidBO1}.}
Notice that $\grp=\gEidBOr{1}$ is a group which contains $\gimShBO$.
Therefore it suffices to verify conditions \coB\ and \coC\ of Theorem~\ref{th:grp_subset_JSh}.

Similarly to~\eqref{equ:init_jets} set $k=\ord(\eta,\orig)$, $\dg_i =\ord(\BFunc_i,\orig)$, $\gamma=\jo{k}(\eta)$, and $\GFunc_i=\jo{\dg_i}(f)$, $(i=1,\ldots,n-1)$.
Denote $\dg=\sum\limits_{i=1}^{n-1}(\dg_i-1)-k$.
Then 
\begin{equation}\label{equ:Q_cross_prod_Gi}
\gamma \cdot \PFunc=[\grad_x\GFunc_1,\ldots,\grad_x\GFunc_{n-1}],
\end{equation}
where $\PFunc=\jo{\dg}(\BFld)$ is a homogeneous vector field of degree $\dg$.
By assumption $\PFunc$ is non-divisible by homogeneous polynomials.
For $p=1$ we assume that $P(x) = \Lmatr x$ for some non-zero matrix $\Lmatr\in M(\RRR,n)$.

Let $\difM \in \gEBO$.
Then $\difM$ leaves invariant every orbit of $\BFld$ and therefore preserves every function $\BFunc_i$, i.e.\! $\BFunc_i\circ\difM=\BFunc_i$ for all $i=1,\ldots,n$.

\coB~Let $A$ be the Jacobi matrix of $\difM$ at $\orig$, thus $\jo1(\difM)(x) = A\cdot x$.
We have to show that $A=e^{\omega_0 L}$ for some $\omega_0\in\RRR$ if $p=1$, and $A=E$ for $p\geq2$.
This is implied by the following lemma and Lemma~\ref{lm:Stab-G1-Gn-1}.
\begin{lemma}
Let $\Stab=\mathop\cap\limits_{i=1}^{n-1}\Stab(\BFunc_i)$ be the intersection of the stabilizers of $\BFunc_i$ with respect to the action of $\GL(\RRR,n)$, see~\eqref{equ:stab_G1_Gn1}, and $\Stab_{E}$ be the unity component of $\Stab$.
Then for every $\difM\in\gEidBOr{1}$ its Jacobi matrix $A$ at $\orig$ belongs to $\Stab_{E}$.
\end{lemma}
\begin{proof}
Since $\BFunc_i\circ\difM=\BFunc_i$, $(i=1,\ldots,n-1)$, we get from Corollary~\ref{cor:deg_Fh-F} that $\GFunc_i(A x)=\GFunc_i(x)$.
In other words $A$ belongs to the intersection of the stabilizers $\Stab = \mathop\cap\limits_{i=1}^{n-1}\Stab(\GFunc_i).$
On the other hand the assumption $\difM\in\gEndIdBzr{1}$ means that there exists a $1$-isotopy $(\difM_t)$ in $\gEndBz$ between $\difM_0=\id_{\RRR^{n}}$ and $\difM_1=\difM$.
Let $A_t$ be the Jacobi matrix of $\difM_t$ at $\orig$. 
Since $(\difM_t)$ is $1$-isotopy, we have that $(A_t)$ continuously depend on $t$.
Moreover, $A_0=E$, whence $A=A_1$ belongs to the unity component $\Stab_{E}$ of $\Stab$.
\end{proof}

\coC~Suppose that $\jo{\dg+l}(\difM)(x) = x + \vfunc(x)$, where $\vfunc$ is a non-zero homogeneous map of degree $\dg+l\geq2$.
Since $\BFunc_i\circ\difM=\BFunc_i$, $(i=1,\ldots,n-1)$, we obtain from Corollary~\ref{cor:deg_Fh-F} that $\langle\grad\GFunc_i,\vfunc\rangle=0$, whence $\vfunc$ is parallel to the cross-product of gradients $\grad\GFunc_i$ and therefore to $\PFunc$.
Since $\PFunc$ is non-divisible, it follows that $k\geq\dg$ and there exists a unique non-zero homogeneous polynomial $\omega_{l}\in\Poly{l}$ such that $\vfunc=\PFunc\cdot \omega_{l}.$

This completes the proof of~\eqref{equ:joSh-joEidBO1}.

It remains to construct $\jo{\infty}$-sections of $\ShBV$.
Let $\Vman$ be a neighbourhood of $\orig$.
Then~\eqref{equ:joSh-joEidBO1} implies that $\EidBVr{1}\subset \JShBV$.
If $\dg\geq2$, then by (1) of Theorem~\ref{th:cont-estimations} there exists a $\jo{\infty}$-section of $\ShBV$ on all of $\JShBV$ and therefore on $\EidBVr{1}$.

Suppose that $\dg=1$, so $\PFunc(x)=\Lmatr x$ is a linear vector field.
Notice that the corresponding one-parametric subgroup $E_{\Lmatr}=\{e^{\Lmatr t}\}_{t\in\RRR}$ is closed in $\GL(\RRR,n)$ as the unity component of the intersection of closed subgroups of $\GL(\RRR,n)$ (stabilizers of $\BFunc_i$).
Then by (2) and (3) of Theorem~\ref{th:cont-estimations} for every $g\in\JShBV$ there exists a local $\jo{\infty}$-section of $\ShBV$.
In particular, this holds for all $g\in\EidBVr{1} \subset \JShBV$.
\qed

\section{Reduced Hamiltonian vector fields}\label{sect:vf_on_r2}
Let $\BFunc:\RRR^2\to\RRR$ be a real homogeneous polynomial in two variables, so we can write
\begin{equation}\label{equ:decomp-for-Q}
\BFunc(x,y) = \prod_{i=1}^{l} L_i^{l_i}(x,y) \cdot \prod_{j=1}^{q} Q_j^{q_j}(x,y),
\end{equation}
where $l,q\geq1$, every $L_i$ is a linear function, every $Q_j$ is a definite (i.e. irreducible over $\RRR$) quadratic form, $L_i/L_{i'}\not=\mathrm{const}$ for $i\not=i'$, and $Q_j/Q_{j'}\not=\mathrm{const}$ for $j\not=j'$.
Then it can easily be shown that the polynomial
$$D = \prod_{i=1}^{l} L_i^{l_i-1} \cdot \prod_{j=1}^{q} Q_j^{q_j-1}$$ 
is the greatest common divisor of its partial derivatives $\BFunc'_x$ and $\BFunc'_y$. 
Hence the following \emph{homogeneous} vector field of degree $\dg=l+2q-1$ on $\RRR^2$ 
$$
\BFld= [\grad\BFunc]/D = - (\BFunc'_{y}/D) \; \frac{\partial }{\partial x}  \; + \; 
(\BFunc'_{x}/D) \; \frac{\partial }{\partial y}
$$
is non-divisible by homogeneous polynomials.
Thus $\BFld$ has property \AST.
We will call $\BFld$ the \emph{reduced Hamiltonian vector field} of $\BFld$.
Notice that $\orig$ is a unique singular point of $\BFld$.

The following theorem improves~\cite[Theorem~3.2]{Maks:Hamvf:2006} which was based on the previous version of this paper.
\begin{theorem}\label{th:HamVF}
Let $\BFunc$ be a real homogeneous polynomial in two variables, $\BFld$ be its reduced Hamiltonian vector field, and $\Vman$ be an open neighbourhood of $\orig$.
Then $$\imShBV=\EidBVr{1}$$
and for every $g\in\EidBVr{1}$ there exist a $C^{1}_{W}$-neighbourhood $\NN_{g}$ in $\EidBVr{1}$ and a $C^{\infty,\infty}_{W,W}$-continuous and preserving smoothness section $\JS:\NN_{g}\to C^{\infty}(\Vman,\RRR)$ of $\ShBV$, i.e. for every $\difM\in\NN_{g}$ 
$$ \difM(x) =\ShBV(\JS(\difM))(x) = \BFlow(x,\JS(\difM)(x)).$$
In particular, it follows from Theorem~\ref{th:param-rigid} and Lemma~\ref{lm:psm_splc} that $\BFld$ is parameter rigid.
\end{theorem}
\begin{proof}
Since $\BFld$ has property \AST, it follows from Theorem~\ref{th:AST_implies_infjets} that for every $g\in\EidBVr{1}$ there exists a $\jo{\infty}$-section defined on some $C^{1}_{W}$-neighbourhood of $g$ in $\EidBVr{1} \subset \EBV\cap\JShBV$.
Therefore by Lemma~\ref{lm:inf_sect_alm_sect} it suffices to construct a $C^{\infty,\infty}_{W,W}$-continuous and preserving smoothness section $\SI:\EBVi\to C^{\infty}(\Vman,\RRR)$ of $\ShBV$.

For the case $D\equiv1$, i.e., when $\BFunc$ has no multiple factors, such a section was constructed in~\cite[Theorem~3.2]{Maks:Hamvf:2006}.
The detailed analysis of the proof shows that \cite[Theorem~3.2]{Maks:Hamvf:2006} uses only the assumption that coordinate functions of $\BFld$ are relatively prime in $\RRR[x,y]$, i.e. that $\BFld$ has property \AST, but not the assumption that $\BFunc$ has no multiple factors.
This implies that the same arguments prove an existence of $\SI$ for arbitrary $\BFunc$.
The details are left for the reader.
\end{proof}

\section{Acknowledgments}
I am sincerely grateful to V.\;V.\;Sharko, V.\;V.\;Lyubashenko, E.\;Po\-lu\-lyakh, A.\;Prish\-lyak, I.\;Vlasenko, I.\;Yurchuk, and O.\;Burylko for useful discussions and interest to this work.
I would like to thank anonymous referee for constructive remarks which allow to clarify exposition of the paper.

 \bigskip
\noindent
{\sc Sergiy Maksymenko} \\
Topology department, Institute of Mathematics of NAS of Ukraine, \\
Tereshchenkivs'ka st. 3, Kyiv, 01601 Ukraine\\
e-mail: \texttt{maks@imath.kiev.ua}

\end{document}